\documentclass[11pt]{article}
\usepackage{amssymb,amsfonts}
\usepackage{amsmath,amsthm,amsxtra}

\usepackage{graphicx}
\usepackage{amscd}
\usepackage{ulem}
\usepackage{makeidx}
\usepackage{fancybox}

\newcommand{\ccc}{{\mathbf C}}

\newcommand{\nnn}{{\mathbf N}}

\newcommand{\rrr}{{\mathbf R}}
\newcommand{\zzz}{{\mathbf Z}}

\newtheorem{thm}{Theorem}[section]
\newtheorem{prop}{Proposition}[section]
\newtheorem{lemma}{Lemma}[section]
\newtheorem{cor}{Corollary}[section]

\newtheorem{note}{Note}[section]

\numberwithin{equation}{section}

\begin{document}

\title{
Mock theta functions and characters of \\ N=3 superconformal modules IV}

\author{\footnote{12-4 Karato-Rokkoudai, Kita-ku, Kobe 651-1334, Japan, 
\hspace{5mm}
wakimoto.minoru.314@m.kyushu-u.ac.jp, \quad wakimoto@r6.dion.ne.jp
}{\, Minoru Wakimoto}}

\date{\empty}

\maketitle

\begin{center}
Abstract
\end{center}

In this paper we obtain explicit formulas for mock theta functions 
$\Phi^{[m,s]}(\tau, z_1, z_2,t)$  $(m \in \frac12 \nnn, s \in \frac12 \zzz)$
by using the coroot lattice of the Lie superalgebra $D(2,1,a)$ 
and the Kac-Peterson's identity.
As its application, we study the branching functions of tensor products 
of N=3 modules and prove the formula conjectured in \cite{W2022d}.

\tableofcontents

\section{Introduction}

In our previous paper \cite{W2022b}, we studied mock theta functions 
$\Phi^{[m,s]}$ by using the coroot lattice of $D(2,1;a)$ 
and the denominator identity of $\widehat{sl}(2|1)$
and deduced explicit formulas for $\Phi^{[m,\frac12]}$.
In the current paper, we study $\Phi^{[m,s]}$ by using the 
Kac-Peterson's identity and obtain explicit formulas 
for $\Phi^{[m,0]}$ and $\Phi^{[m,\frac12]}$ in section \ref{sec:explicit:Phi}.
Since all $\Phi^{[m,s]}$ $(m \in \frac12\nnn, \, s \in \frac12 \zzz)$ 
are obtained from $\Phi^{[m,0]}$ and $\Phi^{[m,\frac12]}$ by 
Lemma 2.8 in \cite{W2022a}, 
Propositions \ref{n3:prop:2022-823a} and \ref{n3:prop:2022-824a}
give explicit formulas for all $\Phi^{[m,s]}$ 
$(m \in \frac12\nnn, \, s \in \frac12 \zzz)$.

\medskip

This paper is organized as follows.

In section \ref{sec:preliminaries}, we deduce the 
Kac-Peterson's identity (Lemma \ref{n3:lemma:2022-823a}) from 
the denominator identity of the affine Lie superalgebra 
$\widehat{osp}(3|2)$.
In section \ref{sec:formula:Phi} we deduce explicit formulas for 
$\Phi^{[\frac{m}{2},s]}(2\tau, 
z+\frac{\tau}{2}-\frac12, \, 
z-\frac{\tau}{2}+\frac12, \, \frac{\tau}{8})$ which are the numerator 
of the N=3 modules by Proposition 4.1 in \cite{W2022a}.
In section \ref{sec:N=3:numerator}, we study the properties of the 
space of N=3 numerators by using theta functions.
In section \ref{sec:N=3:character}, using the results in section \ref{sec:N=3:numerator}, 
we obtain Theorem \ref{n3:thm:2022-827a} which proves Conjecture 5.1 in 
\cite{W2022d}, and show that the character of the N=3 module 
$H(\Lambda^{[K(m), m_2]})$ is a 
$\ccc((q^{\frac12}))$-linear combination of 
$(\theta_{1,1})^j(\theta_{0,1})^{m-j}$ \,\ $(0 \leq j \leq m)$.

\medskip

We remark here that the functions $F^{[m,s]}_i$ $(m \in \frac12 \nnn, 
s \in \{0, \frac12\}, i \in \{1,2\})$ introduced in the proof of Lemmas 
\ref{n3:lemma:2022-823b} and \ref{n3:lemma:2022-824a}
have the Lie theoretic background. In our previous paper \cite{W2022b}, 
we used the coroot lattice 
of $D(2,1,a)$ with $a=\frac{-m}{m+1}$. In the current paper we consider 
the coroot lattice of $D(2,1,a)$ with $a=\frac{-2m}{2m+1}$ and
\begin{equation}\left\{
\begin{array}{lcl}
F_{\frac{a}{2}\Lambda_0+\frac14 \alpha_2} &:=&
\sum\limits_{j, \, k \, \in \zzz}(-1)^j \, 
t_{j\alpha_2^{\vee}+k\alpha_3^{\vee}}\bigg(
\dfrac{e^{\frac{a}{2}\Lambda_0+\frac14 \alpha_2}}{1-e^{-\alpha_1}}\bigg)
\\[6mm]
F_{\frac{a}{2}\Lambda_0+\frac14 (\alpha_2+\alpha_3)} &:=&
\sum\limits_{j, \, k \, \in \zzz}(-1)^j \, 
t_{j\alpha_2^{\vee}+k\alpha_3^{\vee}}\bigg(
\dfrac{e^{\frac{a}{2}\Lambda_0+\frac14 (\alpha_2+\alpha_3)}
}{1-e^{-\alpha_1}}\bigg)
\end{array}\right.
\label{n3:eqn:2022-826a1}
\end{equation}
in the terminology in \S 3 of \cite{W2022b}.
Then the functions $F^{[m,s]}_i$ defined by 
\eqref{n3:eqn:2022-823e1}, \eqref{n3:eqn:2022-823e2}, 
\eqref{n3:eqn:2022-824a1} and \eqref{n3:eqn:2022-824a2} are 
\begin{equation}
\left\{
\begin{array}{lcl}
F^{[m,0]}_1 &=& F_{\frac{a}{2}\Lambda_0+\frac14 \alpha_2} 
\\[2mm]
F^{[m,0]}_2&=& r_3(F_{\frac{a}{2}\Lambda_0+\frac14 \alpha_2}) 
\end{array}\right.  \quad \text{and} \quad 
\left\{
\begin{array}{lcl}
F^{[m,\frac12]}_1 &=& F_{\frac{a}{2}\Lambda_0+\frac14 (\alpha_2+\alpha_3)} 
\\[2mm]
F^{[m,\frac12]}_2&=& r_3(F_{\frac{a}{2}\Lambda_0+\frac14 (\alpha_2+\alpha_3)}) 
\end{array}\right.
\label{n3:eqn:2022-826a2}
\end{equation}
Thus the formulas for $\Phi^{[m,0]}$ and $\Phi^{[m,\frac12]}$ in 
Propositions \ref{n3:prop:2022-823a} and \ref{n3:prop:2022-824a} are 
deduced by calculation on the coroot lattice of $D(2,1; \frac{-2m}{2m+1})$ 
and the Kac-Peterson's identity.

\medskip

For notations and definitions in this paper we follow from \cite{W2022a}
\cite{W2022b}, \cite{W2022c} and \cite{W2022d}.

\section{Preliminaries}
\label{sec:preliminaries}

\begin{lemma} 
\label{n3:lemma:2022-823a}
{\rm (Kac-Peterson's identity: formula (5.26) in \cite{KP}) }  

For $a \, \in \, \zzz$ and $j \in \{1,2\}$, the following formula holds:
\begin{equation}
\Phi^{(-)[\frac12, \frac12]}_j(\tau, \, z, \, -z+2a\tau, \, 0)
\,\ = \,\ 
- \, i \, e^{-2\pi iaz} \, \dfrac{\eta(\tau)^3}{\vartheta_{11}(\tau,z)}
\label{n3:eqn:2022-823a}
\end{equation}
\end{lemma}

\begin{proof} Recall the denominator identity of the affine Lie superalgebra 
$\widehat{osp}(3|2)$ in \cite{KW1994} and \cite{KW2017b}:
\begin{subequations}
{\allowdisplaybreaks
\begin{eqnarray}
& & \hspace{-10mm}
\big[\Phi^{(-)[\frac12, \frac12]}_1
+
\Phi^{(-)[\frac12, \frac12]}_2\big](\tau, z_1, z_2,0) 
\, = \, 
i \,\ \frac{
\eta(\tau)^3 \vartheta_{11}(\tau, z_1+z_2) \, 
\vartheta_{11}\Big(\tau, \dfrac{z_1-z_2}{2}\Big)
}{
\vartheta_{11}(\tau, z_1) \, 
\vartheta_{11}(\tau, z_2) \, 
\vartheta_{11}\Big(\tau, \dfrac{z_1+z_2}{2}\Big)}
\label{n3:eqn:2022-823b}
\\[1mm]
& &
= \,\ i \,\ \frac{\eta(\tau)^2}{\eta(2\tau)} \cdot
\frac{\vartheta_{01}(2\tau, z_1+z_2) \, 
\vartheta_{10}\Big(\tau, \dfrac{z_1+z_2}{2}\Big)
\vartheta_{11}\Big(\tau, \dfrac{z_1-z_2}{2}\Big)
}{
\vartheta_{11}(\tau, z_1) \, 
\vartheta_{11}(\tau, z_2)}
\label{n3:eqn:2022-823b2}
\end{eqnarray}}
\end{subequations}
Letting $z_1=z$ and $z_2=-z+2a\tau$ in this equation 
\eqref{n3:eqn:2022-823b2}, we have 
\begin{subequations}
{\allowdisplaybreaks
\begin{eqnarray}
& & \hspace{-20mm}
\big[\Phi^{(-)[\frac12, \frac12]}_1
+
\Phi^{(-)[\frac12, \frac12]}_2\big](\tau, z, -z+2a\tau, 0)
\nonumber
\\[2mm]
&=&
i \,\ \dfrac{\eta(\tau)^2}{\eta(2\tau)} \cdot 
\frac{
\vartheta_{01}(2\tau, 2a\tau)\vartheta_{10}(\tau, a\tau)
\vartheta_{11}(\tau, z-a\tau)
}{\vartheta_{11}(\tau, z) \vartheta_{11}(\tau, -z+2a\tau)} \hspace{10mm}
\nonumber
\\[1mm]
&=& 
- \, 2i \, e^{-2\pi iaz} \, \dfrac{\eta(\tau)^3}{\vartheta_{11}(\tau,z)}
\label{n3:eqn:2022-823c}
\end{eqnarray}}
since
$$ \hspace{-24.5mm}
\left\{
\begin{array}{lcr}
\vartheta_{10}(\tau, a\tau) &=& 2 \, q^{-\frac12a^2} 
\dfrac{\eta(2\tau)^2}{\eta(\tau)} 
\\[5mm]
\vartheta_{01}(\tau, a\tau) &=& (-1)^a q^{-\frac12a^2} 
\dfrac{\eta(\frac{\tau}{2})^2}{\eta(\tau)}
\end{array}\right.
$$
and
$$
\left\{
\begin{array}{lcrl}
\vartheta_{11}(\tau, z-a\tau) &=& 
(-1)^a q^{-\frac12 a^2}e^{2\pi iaz}\vartheta_{11}(\tau, z) &
\\[3mm]
\vartheta_{11}(\tau, -z+2a\tau) &=& 
-q^{-2 a^2}e^{4\pi iaz}\vartheta_{11}(\tau, z) & .
\end{array}\right.  
$$
Also, by Lemma 2.5 in \cite{W2022c}, we have
\begin{equation}
\Phi^{(-)[\frac12,\frac12]}_1(\tau, z, -z+2a\tau,0) 
=
\Phi^{(-)[\frac12,\frac12]}_2(\tau, z, -z+2a\tau,0) 
\label{n3:eqn:2022-823d}
\end{equation} 
\end{subequations}
Then Lemma \ref{n3:lemma:2022-823a} follows from 
\eqref{n3:eqn:2022-823c} and \eqref{n3:eqn:2022-823d}.
\end{proof}

\vspace{0mm}

Following formulas for theta functions can be shown easily by 
using Lemma 1.1 in \cite{W2022c} and Note 1.1 in \cite{W2022c}, 
and will be used in the proof of Lemma \ref{n3:lemma:2022-825a}.

\medskip

\begin{note} 
\label{n3:note:2022-825a}
For $p \in \zzz$, the following formulas hold:
\begin{enumerate}
\item[{\rm 1)}] \quad $\vartheta_{11}(2\tau, \, z+\frac{\tau}{2}-\frac12+p\tau) 
\,\ = \,\ 
q^{-\frac14(p+\frac12)^2} \, e^{-\pi i(p+\frac12)z}\theta_{p-\frac12, 1}(\tau,z)$
\item[{\rm 2)}] \quad $\vartheta_{11}(2\tau, \, z-\frac{\tau}{2}+\frac12-p\tau) 
\,\ = \,\ - \, 
q^{-\frac14(p+\frac12)^2} \, e^{\pi i(p+\frac12)z}\theta_{-p+\frac12, 1}(\tau,z)$
\end{enumerate}
\end{note}

\vspace{0mm}

\begin{note} 
\label{n3:note:2022-825b}
For $m \in \nnn$ and $p \in \zzz$, the following formulas hold:
\begin{enumerate}
\item[{\rm 1)}]
\begin{enumerate}
\item[{\rm (i)}] \quad $\theta_{-1, m+1}^{(-)}
\Big(\tau, \, z+\dfrac{(2p+1)\tau-1}{2(m+1)}\Big)
\,\ = \,\ 
e^{\frac{\pi i}{2(m+1)}} \, q^{-\frac{(2p+1)^2}{16(m+1)}} \, 
e^{-\frac{\pi i(2p+1)z}{2}} \, 
\theta_{p-\frac12, m+1}(\tau,z)$
\item[{\rm (ii)}] \quad $\theta_{-1, m+1}^{(-)}
\Big(\tau, \, z+\dfrac{-(2p+1)\tau+1}{2(m+1)}\Big)
\,\ = \,\ 
e^{\frac{\pi i}{2(m+1)}} \, q^{-\frac{(2p+1)^2}{16(m+1)}} \, 
e^{\frac{\pi i(2p+1)z}{2}} \, 
\theta_{-p+\frac12, m+1}(\tau,z)$
\end{enumerate}
\item[{\rm 2)}]
\begin{enumerate}
\item[{\rm (i)}] \quad $\theta_{0, m+1}^{(-)}
\Big(\tau, \, z+\dfrac{(2p+1)\tau-1}{2(m+1)}\Big)
\,\ = \,\ 
q^{-\frac{(2p+1)^2}{16(m+1)}} \, 
e^{-\frac{\pi i(2p+1)z}{2}} \, 
\theta_{p+\frac12, m+1}(\tau,z)$
\item[{\rm (ii)}] \quad $\theta_{0, m+1}^{(-)}
\Big(\tau, \, z+\dfrac{-(2p+1)\tau+1}{2(m+1)}\Big)
\,\ = \,\ 
q^{-\frac{(2p+1)^2}{16(m+1)}} \, 
e^{\frac{\pi i(2p+1)z}{2}} \, 
\theta_{-p-\frac12, m+1}(\tau,z)$
\end{enumerate}
\end{enumerate}
\end{note}

\vspace{1mm}

\begin{note} 
\label{n3:note:2022-825c}
For $m \in \nnn$ and $p \in \zzz$, the following formulas hold:
\begin{enumerate}
\item[{\rm 1)}] \,\ $\theta_{1, m+1}^{(-)}\Big(\tau, \, 
\dfrac{m(2p+1)\tau-m}{2(m+1)}\Big) 
\,\ = \,\ 
e^{-\frac{\pi im}{2(m+1)}} \, 
q^{-\frac{m^2}{16(m+1)}(2p+1)^2} \, 
\theta_{1+m(p+\frac12), m+1}^{(\sigma(m))}(\tau,0)$
\item[{\rm 2)}] \,\ $\theta_{0, m+1}^{(-)}\Big(\tau, \, 
\dfrac{m(2p+1)\tau-m}{2(m+1)}\Big) 
\,\ = \,\ 
q^{-\frac{m^2}{16(m+1)}(2p+1)^2} \, 
\theta_{m(p+\frac12), m+1}^{(\sigma(m))}(\tau,0)$
\end{enumerate}
where 
\begin{equation}
\sigma(m) \,\ := \,\ \left\{
\begin{array}{cccl}
- & & {\rm if} & m \, \in \, \nnn_{\rm even} \\[1mm]
+ & & {\rm if} & m \, \in \, \nnn_{\rm odd} 
\end{array}\right. 
\label{n3:eqn:2022-825b}
\end{equation}
\end{note}

\section{Explicit formulas for $\Phi^{[m,s]}$}
\label{sec:explicit:Phi}

\subsection{$\Phi^{[m,0]}$}
\label{subsec:explicit:Phi[m,0]}

\medskip

\begin{lemma}
\label{n3:lemma:2022-823b}
For $m \in \frac12 \nnn$, the following formulas hold:
\begin{subequations}
\begin{enumerate}
\item[{\rm 1)}] $\sum\limits_{j \in \zzz} (-1)^j
q^{(m+\frac12)j^2+\frac12 j} \, 
e^{\pi ij(z_1-z_2)-2\pi ijm(z_1-z_3)} \, 
\Phi^{[m,0]}_1(\tau, \, z_1, \, -z_3-2j\tau, \, 0)$
\begin{eqnarray}
&=& 
e^{-\pi iz_1} \sum_{k \in \zzz} (-1)^k
q^{(m+\frac12)k^2-\frac12 k} \, 
e^{-\pi ik(z_1-z_2)+2\pi ikm(z_1-z_3)} 
\nonumber
\\[1mm]
& & \hspace{15mm}
\times \,\ 
\Phi^{(-)[\frac12,\frac12]}_1(\tau, \, z_1, \, -z_2-2k\tau, \, 0)
\label{n3:eqn:2022-823h1}
\end{eqnarray}

\item[{\rm 2)}] $\sum\limits_{j \in \zzz} (-1)^j
q^{(m+\frac12)j^2+\frac12 j} \, 
e^{\pi ij(z_1-z_2)-2\pi ijm(z_1-z_3)} \, 
\Phi^{[m,0]}_2(\tau, \, z_1, \, -z_3-2j\tau, \, 0)$
{\allowdisplaybreaks
\begin{eqnarray}
&=&
e^{-\pi i z_3} \sum\limits_{k \in \zzz} (-1)^k
q^{(m+\frac12)k^2-\frac12 k} \, 
e^{-\pi ik(z_1-z_2)-2\pi ikm(z_1-z_3)} 
\nonumber
\\[1mm]
& & \hspace{15mm}
\times \,\ 
\Phi^{(-)[\frac12,\frac12]}_1(\tau, \, z_3, \, z_1-z_2-z_3-2k\tau, \, 0)
\label{n3:eqn:2022-823h2}
\end{eqnarray}}
\end{enumerate}
\end{subequations}
\end{lemma}

\begin{proof} To prove this lemma, we consider the following two functions:
\begin{subequations}
{\allowdisplaybreaks
\begin{eqnarray}
F^{[m,0]}_1 &:=& 
e^{\frac{\pi i}{2}(z_1-z_2)}
\sum\limits_{j, \, k \, \in \, \zzz} (-1)^{j} \, 
\dfrac{e^{\pi ij(z_1-z_2)+2\pi ikm(z_1-z_3)} \, q^{\frac12 j^2+mk^2+\frac12 j}
}{1-e^{2\pi iz_1} \, q^{j+k}}
\label{n3:eqn:2022-823e1}
\\[2mm]
F^{[m,0]}_2 &:=& 
e^{\frac{\pi i}{2}(z_1-z_2)}
\sum\limits_{j, \, k \, \in \, \zzz} (-1)^{j} \, 
\dfrac{e^{\pi ij(z_1-z_2)-2\pi ikm(z_1-z_3)} \, q^{\frac12 j^2+mk^2+\frac12 j}
}{1-e^{2\pi iz_3} \, q^{j+k}}
\label{n3:eqn:2022-823e2}
\end{eqnarray}}
\end{subequations}
We compute the RHS's of these functions by putting $j+k=r$ as follows.

\medskip

\noindent
To prove 1), as the 1st step, we compute the RHS of \eqref{n3:eqn:2022-823e1} 
by replacing $(j,k)$ with $(j,r)$: 
\begin{subequations}
{\allowdisplaybreaks
\begin{eqnarray}
& & \hspace{-10mm}
F^{[m,0]}_1 \, = \, 
e^{\frac{\pi i}{2}(z_1-z_2)}
\sum_{j, \, r \, \in \, \zzz} (-1)^j
\frac{e^{\pi ij(z_1-z_2)+2\pi im(r-j)(z_1-z_3)} \, 
q^{\frac12 j^2+ m(r-j)^2 + \frac12 j}
}{1-e^{2\pi iz_1} \, q^r}
\nonumber
\\[3mm]
&=&
e^{\frac{\pi i}{2}(z_1-z_2)}
\sum_{j \in \zzz} (-1)^j \, 
q^{(m+\frac12)j^2+\frac12 j} \, 
e^{\pi ij(z_1-z_2)-2\pi ijm(z_1-z_3)} 
\nonumber
\\[-2mm]
& &\hspace{20mm}
\times \,\ 
\underbrace{\sum_{r \in \zzz} \frac{
e^{2\pi imr(z_1-z_3-2j\tau)}
\, q^{mr^2}}{1-e^{2\pi iz_1} \, q^r}}_{\substack{|| \\[-2mm] 
{\displaystyle \hspace{-2mm}
\Phi^{[m,0]}_1(\tau, \, z_1, \, -z_3-2j\tau, \, 0)
}}}
\label{n3:eqn:2022-823f1}
\end{eqnarray}}

As the 2nd step, we compute the RHS of \eqref{n3:eqn:2022-823e1} 
by replacing $(j,k)$ with $(k,r)$: 
{\allowdisplaybreaks
\begin{eqnarray}
& & \hspace{-10mm}
F^{[m,0]}_1 \, = \, 
e^{\frac{\pi i}{2}(z_1-z_2)}
\sum_{k, \, r \, \in \, \zzz} (-1)^{k+r} 
\dfrac{e^{\pi i(r-k)(z_1-z_2)+2\pi ikm(z_1-z_3)} 
q^{\frac12 (r-k)^2+mk^2+\frac12(r-k)}
}{1-e^{2\pi iz_1} \, q^r}
\nonumber
\\[2mm]
&=&
e^{-\frac{\pi i}{2}(z_1+z_2)}
\sum_{k \in \zzz} (-1)^k
q^{(m+\frac12)k^2-\frac12 k} \, 
e^{-\pi ik(z_1-z_2)+2\pi ikm(z_1-z_3)} 
\nonumber
\\[-2mm] 
& & \hspace{20mm}
\times \,\ 
\underbrace{\sum_{r \in \zzz} (-1)^r \, \frac{
e^{\pi ir(z_1-z_2-2k\tau)}
\, e^{\pi iz_1}
\, q^{\frac12 r^2+\frac12 r}
}{1-e^{2\pi iz_1} \, q^r} }_{\substack{|| \\[-2mm] {\displaystyle 
\Phi^{(-)[\frac12,\frac12]}_1(\tau, \, z_1, \, -z_2-2k\tau, \, 0)
}}}
\label{n3:eqn:2022-823f2}
\end{eqnarray}}
\end{subequations}
Then by \eqref{n3:eqn:2022-823f1} and \eqref{n3:eqn:2022-823f2} 
we have 
{\allowdisplaybreaks
\begin{eqnarray*}
& &
e^{\frac{\pi i}{2}(z_1-z_2)}
\sum\limits_{j \in \zzz} (-1)^j
q^{(m+\frac12)j^2+\frac12 j} \, 
e^{\pi ij(z_1-z_2)-2\pi ijm(z_1-z_3)} \, 
\Phi^{[m,0]}_1(\tau, \, z_1, \, -z_3-2j\tau, \, 0)
\\[1mm]
&=&
e^{-\frac{\pi i}{2}(z_1+z_2)}
\sum_{k \in \zzz} (-1)^k
q^{(m+\frac12)k^2-\frac12 k} \, 
e^{-\pi ik(z_1-z_2)+2\pi ikm(z_1-z_3)} \, 
\Phi^{(-)[\frac12,\frac12]}_1(\tau, \, z_1, \, -z_2-2k\tau, \, 0)
\end{eqnarray*}}
proving 1).

\medskip

In order to prove 2), first we compute the RHS of \eqref{n3:eqn:2022-823e2} 
by replacing $(j,k)$ with $(j,r)$: 
\begin{subequations}
{\allowdisplaybreaks
\begin{eqnarray}
& & \hspace{-10mm}
F^{[m,0]}_2 \, = \, 
e^{\frac{\pi i}{2}(z_1-z_2)}
\sum_{j, \, r \, \in \, \zzz} (-1)^j \, 
\dfrac{e^{\pi ij(z_1-z_2)-2\pi im(r-j)(z_1-z_3)} \, 
q^{\frac12 j^2+ m(r-j)^2 + \frac12 j}}{1-e^{2\pi iz_3} \, q^r}
\nonumber
\\[2mm]
&=&
e^{\frac{\pi i}{2}(z_1-z_2)}
\sum_{j \in \zzz} (-1)^j 
q^{(m+\frac12)j^2+\frac12 j} 
e^{\pi ij(z_1-z_2)+2\pi ijm(z_1-z_3)} 
\underbrace{\sum_{r \in \zzz} \frac{
e^{2\pi imr(-z_1+z_3-2j\tau)} q^{mr^2}}{1-e^{2\pi iz_3} \, q^r}
}_{\substack{|| \\[-1.5mm] 
{\displaystyle \Phi^{[m,0]}_1(\tau, \, z_3, \, -z_1-2j\tau, \, 0)
}}}
\nonumber
\\[2mm]
&=&
e^{\frac{\pi i}{2}(z_1-z_2)}
\sum_{j \in \zzz} (-1)^j
q^{(m+\frac12)j^2+\frac12 j} \, 
e^{\pi ij(z_1-z_2)-2\pi ijm(z_1-z_3)} \, 
\Phi^{[m,0]}_2(\tau, \, z_1, \, -z_3-2j\tau, \, 0)
\nonumber
\\[-2mm]
& &
\label{n3:eqn:2022-823g1}
\end{eqnarray}}
where we used Lemma 2.3 in \cite{W2022b}.

\medskip

Next we compute the RHS of \eqref{n3:eqn:2022-823e2} 
by replacing $(j,k)$ with $(k,r)$: 
{\allowdisplaybreaks
\begin{eqnarray}
& & \hspace{-10mm}
F^{[m,0]}_2 \, = \, 
e^{\frac{\pi i}{2}(z_1-z_2)}
\sum_{k, \, r \, \in \, \zzz} (-1)^{k+r}
\dfrac{e^{\pi i(r-k)(z_1-z_2)-2\pi ikm(z_1-z_3)} \, 
q^{\frac12 (r-k)^2+mk^2+\frac12(r-k)}}{1-e^{2\pi iz_3} \, q^r}
\nonumber
\\[2mm]
&=&
e^{\frac{\pi i}{2}(z_1-z_2-2z_3)}
\sum_{k \in \zzz} (-1)^k
q^{(m+\frac12)k^2-\frac12 k} \, 
e^{-\pi ik(z_1-z_2)-2\pi ikm(z_1-z_3)} 
\nonumber
\\[-2mm]
& & \hspace{20mm}
\times \,\ 
\underbrace{\sum_{r \in \zzz} (-1)^r \, \frac{
e^{\pi ir(z_1-z_2-2k\tau)}\, e^{\pi iz_3}
\, q^{\frac12 r^2+\frac12 r}
}{1-e^{2\pi iz_3} \, q^r} }_{\substack{|| \\[-2mm] 
{\displaystyle 
\Phi^{(-)[\frac12,\frac12]}_1(\tau, \, z_3, \, z_1-z_2-z_3-2k\tau, \, 0)
}}}
\label{n3:eqn:2022-823g2}
\end{eqnarray}}
\end{subequations}
Then by \eqref{n3:eqn:2022-823g1} and \eqref{n3:eqn:2022-823g2} 
we have 
{\allowdisplaybreaks
\begin{eqnarray*}
& &
e^{\frac{\pi i}{2}(z_1-z_2)}
\sum_{j \in \zzz} (-1)^j
q^{(m+\frac12)j^2+\frac12 j} \, 
e^{\pi ij(z_1-z_2)-2\pi ijm(z_1-z_3)} \, 
\Phi^{[m,0]}_2(\tau, \, z_1, \, -z_3-2j\tau, \, 0)
\\[2mm]
&=&
e^{\frac{\pi i}{2}(z_1-z_2-2z_3)}
\sum_{k \in \zzz} (-1)^k
q^{(m+\frac12)k^2-\frac12 k} \, 
e^{-\pi ik(z_1-z_2)-2\pi ikm(z_1-z_3)} 
\nonumber
\\[-2mm]
& & \hspace{35mm}
\times \,\ 
\Phi^{(-)[\frac12,\frac12]}_1(\tau, \, z_3, \, z_1-z_2-z_3-2k\tau, \, 0)
\end{eqnarray*}}
proving 2).
\end{proof}

\vspace{1mm}

\begin{lemma} 
\label{n3:lemma:2022-823c}
For $m \in \frac12 \, \nnn$, the following formula holds:
{\allowdisplaybreaks
\begin{eqnarray}
& & 
e^{\frac{\pi im}{m+\frac12}z_1}
\sum_{j \in \zzz}(-1)^j
q^{(m+\frac12)(j-\frac{1}{4(m+\frac12)})^2}
e^{2\pi im(j-\frac{1}{4(m+\frac12)})(z_1+z_2)}
\Phi^{[m,0]}(\tau, \,\ z_1, \,\ z_2+2j\tau, \,\ 0)
\nonumber
\\[0mm]
&=&
- i \, \eta(\tau)^3 \bigg\{
\frac{ \displaystyle 
\theta_{-\frac12, \, m+\frac12}^{(-)}\Big(\tau, \, z_1+z_2+
\frac{z_1-z_2}{2m+1}\Big)}{\vartheta_{11}(\tau, z_1)}
+ 
\frac{ \displaystyle 
\theta_{-\frac12, \, m+\frac12}^{(-)}\Big(\tau, \, z_1+z_2+
\frac{z_2-z_1}{2m+1}\Big)}{\vartheta_{11}(\tau, z_2)}
\bigg\}
\label{n3:eqn:2022-823m}
\end{eqnarray}}
\end{lemma}

\begin{proof} Computing the difference \lq \lq ${\rm \eqref{n3:eqn:2022-823h1}}
-{\rm \eqref{n3:eqn:2022-823h2}}$", 
we have the following formula for $m \in \frac12 \nnn$:
\begin{subequations}
{\allowdisplaybreaks
\begin{eqnarray}
& & 
\sum_{j \in \zzz} (-1)^j
q^{(m+\frac12)j^2+\frac12 j}
e^{\pi ij(z_1-z_2)-2\pi ijm(z_1-z_3)} \, 
\Phi^{[m,0]}(\tau, \, z_1, \, -z_3-2j\tau, \, 0) 
\nonumber
\\[2mm]
&=& 
e^{-\pi iz_1}\sum_{k \in \zzz} (-1)^k
q^{(m+\frac12)k^2-\frac12 k} 
e^{-\pi ik(z_1-z_2)+2\pi ikm(z_1-z_3)} 
\Phi_1^{(-)[\frac12,\frac12]}(\tau, \, z_1, \, -z_2-2k\tau, \, 0)
\nonumber
\\[1mm]
& & \hspace{-5mm}
- \,\ e^{-\pi iz_3} \sum_{k \in \zzz} (-1)^k
q^{(m+\frac12)k^2-\frac12 k} 
e^{-\pi ik(z_1-z_2)-2\pi ikm(z_1-z_3)} 
\nonumber
\\[1mm]
& & \hspace{20mm}
\times \,\ \Phi_1^{(-)[\frac12,\frac12]}
(\tau, \, z_3, \, z_1-z_2-z_3-2k\tau, \, 0)
\label{n3:eqn:2022-823j}
\end{eqnarray}}
This formula \eqref{n3:eqn:2022-823j} is rewritten as follows:
{\allowdisplaybreaks
\begin{eqnarray}
& &
e^{\frac{\pi im}{m+\frac12}z_1}
\sum_{j \in \zzz}(-1)^j
q^{(m+\frac12)(j-\frac{1}{4(m+\frac12)})^2}
e^{2\pi i(j-\frac{1}{4(m+\frac12)})\{-\frac12 (z_1-z_2)+m(z_1-z_3)\}}
\nonumber
\\[0mm]
& & \hspace{20mm}
\times \,\ \Phi^{[m,0]}(\tau, \,\ z_1, \,\ -z_3+2j\tau, \,\ 0)
\nonumber
\\[3mm]
&=&
e^{ - \frac{\pi i}{2(m+\frac12)} z_1}
\sum_{k \in \zzz} (-1)^k \, 
q^{(m+\frac12)(k+\frac{1}{4(m+\frac12)})^2}
e^{2\pi i(k+\frac{1}{4(m+\frac12)})\{\frac12 (z_1-z_2)-m(z_1-z_3)\}}
\nonumber
\\[0mm]
& & \hspace{20mm}
\times \,\ 
\Phi^{(-)[\frac12, \frac12]}_1
(\tau, \,\ z_1, \,\ -z_2+2k\tau, \,\ 0)
\nonumber
\\[3mm]
& & \hspace{-5mm}
- \,\ e^{ - \frac{\pi i}{2(m+\frac12)} z_3}
\sum_{k \in \zzz} (-1)^k \, 
q^{(m+\frac12)(k+\frac{1}{4(m+\frac12)})^2}
e^{2\pi i(k+\frac{1}{4(m+\frac12)})\{\frac12 (z_1-z_2)+m(z_1-z_3)\}}
\nonumber
\\[0mm]
& & \hspace{20mm}
\times \,\ 
\Phi^{(-)[\frac12, \frac12]}_1
(\tau, \,\ z_3, \,\ z_1-z_2-z_3+2k\tau, \,\ 0)
\label{n3:eqn:2022-823k}
\end{eqnarray}}
\end{subequations}
Letting $z_1=z_2$ in this formula \eqref{n3:eqn:2022-823k} and using 
Lemma \ref{n3:lemma:2022-823a}, we have
{\allowdisplaybreaks
\begin{eqnarray}
& &
e^{\frac{\pi im}{m+\frac12}z_1}
\sum_{j \in \zzz}(-1)^j
q^{(m+\frac12)(j-\frac{1}{4(m+\frac12)})^2}
e^{2\pi im(j-\frac{1}{4(m+\frac12)})(z_1-z_3)}
\Phi^{[m,0]}(\tau, \, z_1, \, -z_3+2j\tau, \, 0)
\nonumber
\\[3mm]
&=&
e^{-\frac{\pi i}{2(m+\frac12)} z_1}
\sum_{k \in \zzz} (-1)^k \, 
q^{(m+\frac12)(k+\frac{1}{4(m+\frac12)})^2}
e^{-2\pi im(k+\frac{1}{4(m+\frac12)})(z_1-z_3)}
\nonumber
\\[0mm]
& & \hspace{30mm}
\times \,\ 
\underbrace{\Phi^{(-)[\frac12, \frac12]}_1
(\tau, \,\ z_1, \,\ -z_1+2k\tau, \,\ 0)}_{
\substack{|| \\[-2mm] {\displaystyle 
- \, i \, e^{-2\pi ikz_1}\frac{\eta(\tau)^3}{\vartheta_{11}(\tau, z_1)}
}}}
\nonumber
\\[3mm]
& &
- \,\ e^{ - \frac{\pi i}{2(m+\frac12)} z_3}
\sum_{k \in \zzz} (-1)^k \, 
q^{(m+\frac12)(k+\frac{1}{4(m+\frac12)})^2}
e^{2\pi im(k+\frac{1}{4(m+\frac12)})(z_1-z_3)}
\nonumber
\\[0mm]
& & \hspace{30mm}
\times \,\ 
\underbrace{
\Phi^{(-)[\frac12, \frac12]}_1
(\tau, \,\ z_3, \,\ -z_3+2k\tau, \,\ 0)}_{
\substack{|| \\[-2mm] {\displaystyle 
- \, i \, e^{-2\pi ikz_3}\frac{\eta(\tau)^3}{\vartheta_{11}(\tau, z_3)}
}}}
\label{n3:eqn:2022-823n}
\end{eqnarray}}
The RHS of this equation \eqref{n3:eqn:2022-823n} is rewritten as follows:
{\allowdisplaybreaks
\begin{eqnarray*}
& & \hspace{-10mm}
\text{RHS of \eqref{n3:eqn:2022-823n}} 
\\[1mm]
&=&
- \, i \, 
\frac{\eta(\tau)^3}{\vartheta_{11}(\tau, z_1)}
\sum_{k \in \zzz} (-1)^k \, 
q^{(m+\frac12)(k+\frac{1}{4(m+\frac12)})^2}
e^{-2\pi im(k+\frac{1}{4(m+\frac12)})(z_1-z_3)}
e^{-2\pi ikz_1} \, e^{-\frac{\pi i}{2(m+\frac12)} z_1}
\\[2mm]
& & 
+ \, i \, 
\frac{\eta(\tau)^3}{\vartheta_{11}(\tau, z_3)}
\sum_{k \in \zzz} (-1)^k \, 
q^{(m+\frac12)(k+\frac{1}{4(m+\frac12)})^2}
e^{2\pi im(k+\frac{1}{4(m+\frac12)})(z_1-z_3)}
e^{-2\pi ikz_3} \, e^{-\frac{\pi i}{2(m+\frac12)} z_3}
\\[2mm]
&=&
- \, i \,\ 
\frac{\eta(\tau)^3}{\vartheta_{11}(\tau, z_1)}
\underbrace{\sum_{k \in \zzz} (-1)^k \, 
q^{(m+\frac12)(k+\frac{1}{4(m+\frac12)})^2}
e^{2\pi i(m+\frac12)(k+\frac{1}{4(m+\frac12)}) \cdot 
\frac{-m(z_1-z_3)-z_1}{m+\frac12}} }_{
\substack{|| \\[0mm] {\displaystyle 
\theta_{\frac12, m+\frac12}^{(-)}\Big(\tau, \, 
\frac{-m(z_1-z_3)-z_1}{m+\frac12}\Big)
}}}
\\[2mm]
& &
+ i \,\ 
\frac{\eta(\tau)^3}{\vartheta_{11}(\tau, z_3)}
\underbrace{\sum_{k \in \zzz} (-1)^k \, 
q^{(m+\frac12)(k+\frac{1}{4(m+\frac12)})^2}
e^{2\pi i(m+\frac12)(k+\frac{1}{4(m+\frac12)}) \cdot 
\frac{m(z_1-z_3)-z_3}{m+\frac12}} }_{
\substack{|| \\[0mm] {\displaystyle 
\theta_{\frac12, m+\frac12}^{(-)}\Big(\tau, \, 
\frac{m(z_1-z_3)-z_3}{m+\frac12}\Big)
}}}
\end{eqnarray*}}
Then the equation \eqref{n3:eqn:2022-823n} becomes as follows:
{\allowdisplaybreaks
\begin{eqnarray*}
& & 
e^{\frac{\pi im}{m+\frac12}z_1}
\sum_{j \in \zzz}(-1)^j
q^{(m+\frac12)(j-\frac{1}{4(m+\frac12)})^2}
e^{2\pi im(j-\frac{1}{4(m+\frac12)})(z_1-z_3)}
\Phi^{[m,0]}(\tau, \,\ z_1, \,\ -z_3+2j\tau, \,\ 0)
\\[2mm]
&=&
-i \, \frac{\eta(\tau)^3}{\vartheta_{11}(\tau, z_1)}  
\theta_{-\frac12, m+\frac12}^{(-)}\Big(\tau,  
\frac{m(z_1-z_3)+z_1}{m+\frac12}\Big)
+ \, i \, 
\frac{\eta(\tau)^3}{\vartheta_{11}(\tau, z_3)} 
\theta_{\frac12, m+\frac12}^{(-)}\Big(\tau, 
\frac{m(z_1-z_3)-z_3}{m+\frac12}\Big)
\end{eqnarray*}}
Changing the notation $-z_3 \rightarrow z_2$, we obtain the formula 
\eqref{n3:eqn:2022-823m}, proving Lemma \ref{n3:lemma:2022-823c}.
\end{proof}

\vspace{1mm}

\begin{prop} 
\label{n3:prop:2022-823a}
For $m \in \frac12 \nnn$, the following formula holds:
{\allowdisplaybreaks
\begin{eqnarray}
& &
\theta_{\frac12, \, m+\frac12}^{(-)}
\Big(\tau, \, \frac{m(z_1-z_2)}{m+\frac12}\Big) \, 
\Phi^{[m,0]}(\tau, \, z_1, \, z_2, \, 0)
\nonumber
\\[2mm]
&=&
\Big[\sum_{\substack{j, \, k \, \in \zzz \\[1mm] 0 \, < \, k \, \leq \, 2mj}}
-
\sum_{\substack{j, \, k \, \in \zzz \\[1mm] 2mj \, < \, k \, \leq \, 0}} \Big] \, 
(-1)^j \, q^{(m+\frac12)(j+\frac{1}{4(m+\frac12)})^2}
e^{2\pi im(j+\frac{1}{4(m+\frac12)})(z_1-z_2)}
\nonumber
\\[1mm]
& & \hspace{30mm}
\times \,\ 
e^{-\pi ik(z_1-z_2)} q^{-\frac{k^2}{4m}} \, 
\big[\theta_{k,m}-\theta_{-k,m}\big](\tau, z_1+z_2)
\nonumber
\\[2mm]
& & \hspace{-7mm}
- \,\ i \, \eta(\tau)^3 \, \Bigg\{
\frac{\displaystyle 
\theta_{-\frac12, \, m+\frac12}^{(-)}\Big(\tau, \, 
z_1+z_2+\frac{z_1-z_2}{2m+1}\Big)}{
\vartheta_{11}(\tau, z_1)}
\, + \, 
\frac{\displaystyle 
\theta_{\frac12, \, m+\frac12}^{(-)}\Big(\tau, \, 
z_1+z_2+\frac{z_2-z_1}{2m+1}\Big)
}{\vartheta_{11}(\tau, z_2)} \Bigg\}
\label{n3:eqn:2022-823p}
\end{eqnarray}}
\end{prop}

\begin{proof} 
We compute the LHS of the equation \eqref{n3:eqn:2022-823m}:
$$
\text{LHS of \eqref{n3:eqn:2022-823m}} \,\ = \,\ 
{\rm (I)}_+ \, + \, {\rm (I)}_-
$$
where
{\allowdisplaybreaks
\begin{eqnarray*}
& &
{\rm (I)}_+ := 
e^{\frac{\pi im}{m+\frac12}z_1} \sum_{j \in \zzz_{\geq 0}}
(-1)^j
q^{(m+\frac12)(j-\frac{1}{4(m+\frac12)})^2}
e^{2\pi im(j-\frac{1}{4(m+\frac12)})(z_1+z_2)}
\Phi^{[m,0]}(\tau, \, z_1, \, z_2+2j\tau, \, 0)
\\[1mm]
& &
{\rm (I)}_- := 
e^{\frac{\pi im}{m+\frac12}z_1} \sum_{j \in \zzz_{< 0}}
(-1)^j
q^{(m+\frac12)(j-\frac{1}{4(m+\frac12)})^2}
e^{2\pi im(j-\frac{1}{4(m+\frac12)})(z_1+z_2)}
\Phi^{[m,0]}(\tau, \, z_1, \, z_2+2j\tau, \, 0)
\end{eqnarray*}}
These functions ${\rm (I)}_{\pm}$ are computed by using Lemma 2.5 
in \cite{W2022b} as follows:
{\allowdisplaybreaks
\begin{eqnarray*}
& & \hspace{-10mm}
{\rm (I)}_+ \,\ = \,\ 
\sum_{j \in \zzz_{\geq 0}}(-1)^j
q^{(m+\frac12)(j-\frac{1}{4(m+\frac12)})^2}
e^{2\pi im(j-\frac{1}{4(m+\frac12)})(z_1+z_2)}
e^{\frac{\pi im}{m+\frac12}z_1} e^{-4\pi imjz_1}
\\[1mm]
& & 
\times \,\ \bigg\{
\Phi^{[m,0]}(\tau, \,\ z_1, \,\ z_2, \,\ 0)
\, - \, 
\sum_{k=0}^{2mj-1}
e^{\pi ik(z_1-z_2)} q^{-\frac{k^2}{4m}} \, 
\big[\theta_{k,m}-\theta_{-k,m}\big](\tau, z_1+z_2)\bigg\}
\\[2mm]
&=&
\sum_{j \in \zzz_{\geq 0}}(-1)^j
q^{(m+\frac12)(j-\frac{1}{4(m+\frac12)})^2}
e^{2\pi im(j-\frac{1}{4(m+\frac12)})(z_2-z_1)}
\Phi^{[m,0]}(\tau, \,\ z_1, \,\ z_2, \,\ 0)
\\[1mm]
& &- \,\ 
\sum_{j \in \zzz_{>0}} \, \sum_{k=0}^{2mj-1} \, (-1)^j
q^{(m+\frac12)(j-\frac{1}{4(m+\frac12)})^2}
e^{2\pi im(j-\frac{1}{4(m+\frac12)})(z_2-z_1)}
\\[-1mm]
& & \hspace{30mm}
\times \,\ 
e^{\pi ik(z_1-z_2)} q^{-\frac{k^2}{4m}} \, 
\big[\theta_{k,m}-\theta_{-k,m}\big](\tau, z_1+z_2)
\\[3mm]
& & \hspace{-10mm}
{\rm (I)}_- \,\ = \,\ 
\sum_{j \in \zzz_{<0}}(-1)^j
q^{(m+\frac12)(j-\frac{1}{4(m+\frac12)})^2}
e^{2\pi im(j-\frac{1}{4(m+\frac12)})(z_1+z_2)}
e^{\frac{\pi im}{m+\frac12}z_1} e^{-4\pi imjz_1}
\\[1mm]
& & 
\times \,\ \bigg\{
\Phi^{[m,0]}(\tau, \,\ z_1, \,\ z_2, \,\ 0)
\, + \, 
\sum_{\substack{k \in \zzz \\[1mm] 2mj \leq k <0}}
e^{\pi ik(z_1-z_2)} q^{-\frac{k^2}{4m}} \, 
\big[\theta_{k,m}-\theta_{-k,m}\big](\tau, z_1+z_2)\bigg\}
\\[2mm]
&=&
\sum_{j \in \zzz_{< 0}}(-1)^j
q^{(m+\frac12)(j-\frac{1}{4(m+\frac12)})^2}
e^{2\pi im(j-\frac{1}{4(m+\frac12)})(z_2-z_1)}
\Phi^{[m,0]}(\tau, \,\ z_1, \,\ z_2, \,\ 0)
\\[2mm]
& &+ \,\ 
\sum_{j < 0} \, 
\sum_{\substack{k \in \zzz \\[1mm] 2mj \, \leq \, k \, < \, 0}} \, 
(-1)^j \, q^{(m+\frac12)(j-\frac{1}{4(m+\frac12)})^2}
e^{2\pi im(j-\frac{1}{4(m+\frac12)})(z_2-z_1)}
\\[0mm]
& & \hspace{30mm}
\times \,\ 
e^{\pi ik(z_1-z_2)} q^{-\frac{k^2}{4m}} \, 
\big[\theta_{k,m}-\theta_{-k,m}\big](\tau, z_1+z_2)
\end{eqnarray*}}
Then we have 
\begin{subequations}
{\allowdisplaybreaks
\begin{eqnarray}
& & \hspace{-17mm}
{\rm (I)}_+ + {\rm (I)}_-  \, = \, 
\underbrace{\sum_{j\in \zzz}(-1)^j
q^{(m+\frac12)(j-\frac{1}{4(m+\frac12)})^2} 
e^{2\pi im(j-\frac{1}{4(m+\frac12)})(z_2-z_1)}
}_{\substack{|| \\[-1mm] {\displaystyle \hspace{3mm}
\theta_{-\frac12, \, m+\frac12}^{(-)}
\Big(\tau, \, \frac{m(z_2-z_1)}{m+\frac12}\Big)
}}}
\Phi^{[m,0]}(\tau, \,\ z_1, \,\ z_2, \,\ 0)
\nonumber
\\[2mm]
& &- \,\  
\Big[\sum_{\substack{j,k \in \zzz \\[1mm] 0 \leq \, k \, < \, 2mj}}
-
\sum_{\substack{j,k \in \zzz \\[1mm] 2mj \, \leq \, k \, < \, 0}}
\Big] \, 
(-1)^j \, q^{(m+\frac12)(j-\frac{1}{4(m+\frac12)})^2}
e^{2\pi im(j-\frac{1}{4(m+\frac12)})(z_2-z_1)}
\nonumber
\\[1mm]
& & \hspace{30mm}
\times \,\ 
e^{\pi ik(z_1-z_2)} q^{-\frac{k^2}{4m}} \, 
\big[\theta_{k,m}-\theta_{-k,m}\big](\tau, z_1+z_2)
\label{n3:eqn:2022-823r1}
\end{eqnarray}}
Putting $j=-j'$ and $k=-k'$, the RHS of this equation 
\eqref{n3:eqn:2022-823r1} is rewitten as follows:
{\allowdisplaybreaks
\begin{eqnarray}
& & \hspace{-10mm}
\text{RHS of \eqref{n3:eqn:2022-823r1}}
\, = \,\ 
\theta_{-\frac12, \, m+\frac12}^{(-)}
\Big(\tau, \, \frac{m(z_2-z_1)}{m+\frac12}\Big)
\Phi^{[m,0]}(\tau, \,\ z_1, \,\ z_2, \,\ 0)
\nonumber
\\[2mm]
&- & \Big[
\sum_{0 \, < \, k' \, \leqq \, 2mj'}
-
\sum_{2mj' \, < \, k' \, \leq \, 0} \Big] \, 
(-1)^{j'} \, q^{(m+\frac12)(j'+\frac{1}{4(m+\frac12)})^2}
e^{2\pi im(j'+\frac{1}{4(m+\frac12)})(z_1-z_2)}
\nonumber
\\[1mm]
& & \hspace{30mm}
\times \,\ 
e^{-\pi ik'(z_1-z_2)} q^{-\frac{k'{}^2}{4m}} \, 
\big[\theta_{k',m}-\theta_{-k',m}\big](\tau, z_1+z_2)
\label{n3:eqn:2022-823r2}
\end{eqnarray}}
\end{subequations}
Then, by \eqref{n3:eqn:2022-823m} and \eqref{n3:eqn:2022-823r2} 
we obtain the equation \eqref{n3:eqn:2022-823p}, proving 
Proposition \ref{n3:prop:2022-823a}.
\end{proof}

\subsection{$\Phi^{[m,\frac12]}$}
\label{subsec:explicit:Phi[m,1/2]}

\begin{lemma}
\label{n3:lemma:2022-824a}
For $m \in \frac12 \nnn$, the following formulas hold:
\begin{subequations}
\begin{enumerate}
\item[{\rm 1)}] $\sum\limits_{j \in \zzz} (-1)^j \, 
q^{(m+\frac12)j^2} \, 
e^{\pi ij(z_1-z_2)-2\pi ijm(z_1-z_3)} \, 
\Phi^{[m,\frac12]}_1(\tau, \, z_1, \, -z_3-2j\tau, \, 0)$
\begin{eqnarray}
&=&
\sum_{k \in \zzz} (-1)^k
q^{(m+\frac12)k^2} \, 
e^{-\pi ik(z_1-z_2)+2\pi ikm(z_1-z_3)} \, 
\Phi^{(-)[\frac12,\frac12]}_1(\tau, \, z_1, \, -z_2-2k\tau, \, 0)
\nonumber
\\[-2mm]
& &
\label{n3:eqn:2022-824d1}
\end{eqnarray}
\item[{\rm 2)}] $\sum\limits_{j \in \zzz} (-1)^j
q^{(m+\frac12)j^2} \, 
e^{\pi ij(z_1-z_2)-2\pi ijm(z_1-z_3)} \, 
\Phi^{[m,\frac12]}_2(\tau, \, z_1, \, -z_3-2j\tau, \, 0)$
\begin{eqnarray}
&=&
\sum_{k \in \zzz} (-1)^k
q^{(m+\frac12)k^2} 
e^{-\pi ik(z_1-z_2)-2\pi ikm(z_1-z_3)} 
\Phi^{(-)[\frac12,\frac12]}_1(\tau, \, z_3, \, z_1-z_2-z_3-2k\tau, \, 0)
\nonumber
\\[-2mm]
& &
\label{n3:eqn:2022-824d2}
\end{eqnarray}
\end{enumerate}
\end{subequations}
\end{lemma}

\begin{proof} To prove this lemma, we consider the following two functions:
\begin{subequations}
{\allowdisplaybreaks
\begin{eqnarray}
F^{[m,\frac12]}_1 &:=& 
e^{\frac{\pi i}{2}(2z_1-z_2-z_3)}
\sum_{j, \, k \, \in \, \zzz} (-1)^{j} \, 
\dfrac{e^{\pi ij(z_1-z_2)+2\pi ikm(z_1-z_3)} \, 
q^{\frac12 j^2+mk^2+\frac12 j+\frac12 k}
}{1-e^{2\pi iz_1} \, q^{j+k}}
\label{n3:eqn:2022-824a1}
\\[1mm]
F^{[m,\frac12]}_2 &:=& 
e^{-\frac{\pi i}{2}(z_2-z_3)}
\sum\limits_{j, \, k \, \in \, \zzz} (-1)^{j} \, 
\dfrac{e^{\pi ij(z_1-z_2)-2\pi ikm(z_1-z_3)} \, 
q^{\frac12 j^2+mk^2+\frac12 j+\frac12 k}
}{1-e^{2\pi iz_3} \, q^{j+k}}
\label{n3:eqn:2022-824a2}
\end{eqnarray}}
\end{subequations}
We compute the RHS's of these functions by putting $j+k=r$ and 
replacing the sum $\sum_{j,k}$ with $\sum_{j,r}$ or $\sum_{k,r}$.

\medskip

\noindent
To prove 1), first we compute the RHS of 
\eqref{n3:eqn:2022-824a1} by replacing the sum $\sum_{j,k}$ with $\sum_{j,r}$ :
\begin{subequations}
{\allowdisplaybreaks
\begin{eqnarray}
& & \hspace{-16mm}
F^{[m,\frac12]}_1 \, = \, 
e^{\frac{\pi i}{2}(2z_1-z_2-z_3)}
\sum_{j, \, r \, \in \, \zzz} (-1)^j
\frac{e^{\pi ij(z_1-z_2)+2\pi im(r-j)(z_1-z_3)} \, 
q^{\frac12 j^2+ m(r-j)^2 + \frac12 r}}{1-e^{2\pi iz_1} \, q^r}
\nonumber
\\[3mm]
&=&
e^{-\frac{\pi i}{2}(z_2+z_3)}
\sum_{j \in \zzz} (-1)^j 
q^{(m+\frac12)j^2} 
e^{\pi ij(z_1-z_2)-2\pi ijm(z_1-z_3)} 
\nonumber
\\[-1mm]
& &\hspace{20mm}
\times \,\ 
\underbrace{\sum_{r \in \zzz} 
\frac{e^{2\pi imr(z_1-z_3-2j\tau)}
e^{\pi iz_1}
q^{mr^2+\frac12 r}}{1-e^{2\pi iz_1} \, q^r}}_{\substack{|| \\[-1mm] 
{\displaystyle \Phi^{[m,\frac12]}_1(\tau, \, z_1, \, -z_3-2j\tau, \, 0)
}}}
\label{n3:eqn:2022-824b1}
\end{eqnarray}}

Next we compute the RHS of \eqref{n3:eqn:2022-824a1} by 
replacing the sum $\sum_{j,k}$ with $\sum_{k,r}$ :
{\allowdisplaybreaks
\begin{eqnarray}
& & \hspace{-10mm}
F^{[m,\frac12]}_1 \, = \, 
e^{\frac{\pi i}{2}(2z_1-z_2-z_3)}
\sum_{k, \, r \, \in \, \zzz} (-1)^{k+r} \, 
\frac{e^{\pi i(r-k)(z_1-z_2)+2\pi ikm(z_1-z_3)} \, 
q^{\frac12 (r-k)^2+mk^2+\frac12 r}}{1-e^{2\pi iz_1} \, q^r}
\nonumber
\\[2mm]
&=&
e^{-\frac{\pi i}{2}(z_2+z_3)}
\sum_{k \in \zzz} (-1)^k
q^{(m+\frac12)k^2} \, 
e^{-\pi ik(z_1-z_2)+2\pi ikm(z_1-z_3)} 
\nonumber
\\[-1mm] 
& & \hspace{20mm}
\times \,\ 
\underbrace{\sum_{r \in \zzz} (-1)^r \, 
\frac{e^{\pi ir(z_1-z_2-2k\tau)} \, e^{\pi iz_1}
\, q^{\frac12 r^2+\frac12 r}
}{1-e^{2\pi iz_1} \, q^r} }_{\substack{|| \\[-1mm] {\displaystyle 
\Phi^{(-)[\frac12,\frac12]}_1(\tau, \, z_1, \, -z_2-2k\tau, \, 0)
}}}
\label{n3:eqn:2022-824b2}
\end{eqnarray}}
\end{subequations}
Then by \eqref{n3:eqn:2022-824b1} and \eqref{n3:eqn:2022-824b2} we have 
{\allowdisplaybreaks
\begin{eqnarray*}
& &
e^{-\frac{\pi i}{2}(z_2+z_3)}
\sum_{j \in \zzz} (-1)^j \, 
q^{(m+\frac12)j^2} \, 
e^{\pi ij(z_1-z_2)-2\pi ijm(z_1-z_3)} \, 
\Phi^{[m,\frac12]}_1(\tau, \, z_1, \, -z_3-2j\tau, \, 0)
\\[2mm]
&=&
e^{-\frac{\pi i}{2}(z_2+z_3)}
\sum_{k \in \zzz} (-1)^k
q^{(m+\frac12)k^2}  
e^{-\pi ik(z_1-z_2)+2\pi ikm(z_1-z_3)}  
\Phi^{(-)[\frac12,\frac12]}_1(\tau, \, z_1, \, -z_2-2k\tau, \, 0)
\end{eqnarray*}}
proving 1). 
\medskip

\noindent
To prove 2), first we compute the RHS of 
\eqref{n3:eqn:2022-824a2} by replacing the sum $\sum_{j,k}$ with $\sum_{j,r}$ :
\begin{subequations}
{\allowdisplaybreaks
\begin{eqnarray}
& & \hspace{-10mm}
F^{[m, \frac12]}_2 \, = \, 
e^{-\frac{\pi i}{2}(z_2-z_3)}
\sum_{j, \, r \, \in \, \zzz} (-1)^j \, 
\dfrac{e^{\pi ij(z_1-z_2)-2\pi im(r-j)(z_1-z_3)} \, 
q^{\frac12 j^2+ m(r-j)^2 + \frac12 r}}{1-e^{2\pi iz_3} \, q^r}
\nonumber
\\[2mm]
&=&
e^{-\frac{\pi i}{2}(z_2+z_3)}
\sum_{j \in \zzz} (-1)^j 
q^{(m+\frac12)j^2} \, 
e^{\pi ij(z_1-z_2)+2\pi ijm(z_1-z_3)} 
\nonumber
\\[2mm]
& & \hspace{20mm}
\times \,\ 
\underbrace{\sum_{r \in \zzz} \frac{
e^{2\pi imr(-z_1+z_3-2j\tau)} \, e^{\pi iz_3}
\, q^{mr^2+\frac12 r}}{1-e^{2\pi iz_3} \, q^r}}_{\substack{|| \\[-1.5mm] 
{\displaystyle \Phi^{[m,\frac12]}_1(\tau, \, z_3, \, -z_1-2j\tau, \, 0)
}}}
\nonumber
\\[3mm]
&=&
e^{-\frac{\pi i}{2}(z_2+z_3)}
\sum_{j \in \zzz} (-1)^j
q^{(m+\frac12)j^2} \, 
e^{\pi ij(z_1-z_2)-2\pi ijm(z_1-z_3)} \, 
\Phi^{[m,\frac12]}_2(\tau, \, z_1, \, -z_3-2j\tau, \, 0)
\nonumber
\\[-2mm]
& &
\label{n3:eqn:2022-824c1}
\end{eqnarray}}
where we used Lemma 2.3 in \cite{W2022b}.

\medskip

Next we compute the RHS of \eqref{n3:eqn:2022-824a2} by replacing 
the sum $\sum_{j,k}$ with $\sum_{k,r}$ :
{\allowdisplaybreaks
\begin{eqnarray}
& & \hspace{-15mm}
F^{[m, \frac12]}_2 \, = \, 
e^{-\frac{\pi i}{2}(z_2-z_3)}
\sum_{k, \, r \, \in \, \zzz} (-1)^{k+r}
\dfrac{e^{\pi i(r-k)(z_1-z_2)-2\pi ikm(z_1-z_3)} \, 
q^{\frac12(r-k)^2+mk^2+\frac12r}}{1-e^{2\pi iz_3} \, q^r}
\nonumber
\\[2mm]
&=&
e^{-\frac{\pi i}{2}(z_2+z_3)}
\sum_{k \in \zzz} (-1)^k
q^{(m+\frac12)k^2} \, 
e^{-\pi ik(z_1-z_2)-2\pi ikm(z_1-z_3)} 
\nonumber
\\[-2mm]
& & \hspace{20mm}
\times \,\ 
\underbrace{\sum_{r \in \zzz} (-1)^r \, \frac{
e^{\pi ir(z_1-z_2-2k\tau)}
\, e^{\pi iz_3}
\, q^{\frac12 r^2+\frac12 r}
}{1-e^{2\pi iz_3} \, q^r} }_{\substack{|| \\[-1mm] 
{\displaystyle 
\Phi^{(-)[\frac12,\frac12]}_1(\tau, \, z_3, \, z_1-z_2-z_3-2k\tau, \, 0)
}}}
\label{n3:eqn:2022-824c2}
\end{eqnarray}}
\end{subequations}
Then by \eqref{n3:eqn:2022-824c1} and \eqref{n3:eqn:2022-824c2} we have 
{\allowdisplaybreaks
\begin{eqnarray*}
& &
e^{-\frac{\pi i}{2}(z_2+z_3)}
\sum_{j \in \zzz} (-1)^j
q^{(m+\frac12)j^2} \, 
e^{\pi ij(z_1-z_2)-2\pi ijm(z_1-z_3)} \, 
\Phi^{[m,\frac12]}_2(\tau, \, z_1, \, -z_3-2j\tau, \, 0)
\\[2mm]
&=&
e^{-\frac{\pi i}{2}(z_2+z_3)}
\sum_{k \in \zzz} (-1)^k
q^{(m+\frac12)k^2} \, 
e^{-\pi ik(z_1-z_2)-2\pi ikm(z_1-z_3)} 
\\[0mm]
& & \hspace{20mm}
\times \,\ 
\Phi^{(-)[\frac12,\frac12]}_1(\tau, \, z_3, \, z_1-z_2-z_3-2k\tau, \, 0)
\end{eqnarray*}}
proving 2).
\end{proof}

\vspace{1mm}

\begin{lemma} \quad 
\label{n3:lemma:2022-824b}
For $m \in \frac12 \nnn$, the following formula holds:
{\allowdisplaybreaks
\begin{eqnarray}
& & 
\sum_{j \in \zzz}(-1)^j
q^{(m+\frac12)j^2}
e^{2\pi imj(z_1+z_2)}
\Phi^{[m,\frac12]}(\tau, \,\ z_1, \,\ z_2+2j\tau, \,\ 0)
\nonumber
\\[0mm]
&=&
- \, i \, \eta(\tau)^3 \bigg\{
\frac{ \displaystyle 
\theta_{0, \, m+\frac12}^{(-)}\Big(\tau, \, z_1+z_2+
\frac{z_1-z_2}{2m+1}\Big)}{\vartheta_{11}(\tau, z_1)}
+ 
\frac{ \displaystyle 
\theta_{0, \, m+\frac12}^{(-)}\Big(\tau, \, z_1+z_2+
\frac{z_2-z_1}{2m+1}\Big)}{\vartheta_{11}(\tau, z_2)}
\bigg\}
\label{n3:eqn:2022-824e}
\end{eqnarray}}
\end{lemma}

\begin{proof} Letting $(j,k) \rightarrow (-j,-k)$ in 
\lq \lq ${\rm \eqref{n3:eqn:2022-824d1}}- 
{\rm \eqref{n3:eqn:2022-824d2}}$" , we have the following formula 
for $m \in \frac12 \nnn$: 
{\allowdisplaybreaks
\begin{eqnarray}
& & 
\sum_{j \in \zzz} (-1)^j q^{(m+\frac12)j^2}
e^{2\pi ij[-\frac12(z_1-z_2)+m(z_1-z_3)]}
\Phi^{[m,\frac12]}(\tau, \, z_1, \, -z_3+2j\tau, \, 0) 
\nonumber
\\[1mm]
&=& 
\sum_{k \in \zzz} (-1)^k q^{(m+\frac12)k^2} 
e^{2\pi ik[\frac12(z_1-z_2)-m(z_1-z_3)]}
\Phi_1^{(-)[\frac12,\frac12]}
(\tau, \, z_1, \, -z_2+2k\tau, \, 0)
\nonumber
\\[1mm]
& & \hspace{-5mm}
- \, 
\sum_{k \in \zzz} (-1)^k q^{(m+\frac12)k^2} 
e^{2\pi ik[\frac12(z_1-z_2)+m(z_1-z_3)]}
\Phi_1^{(-)[\frac12,\frac12]}
(\tau, \, z_3, \, z_1-z_2-z_3+2k\tau, \, 0)
\nonumber
\\[-2mm]
& &
\label{n3:eqn:2022-824f}
\end{eqnarray}}
Letting $z_1=z_2$ in this formula \eqref{n3:eqn:2022-824f}
and using Lemma \ref{n3:lemma:2022-823a}, we have 
{\allowdisplaybreaks
\begin{eqnarray*}
& &
\sum_{j \in \zzz}(-1)^j
q^{(m+\frac12)j^2}
e^{2\pi imj(z_1-z_3)}
\Phi^{[m,\frac12]}(\tau, \, z_1, \, -z_3+2j\tau, \, 0)
\nonumber
\\[0mm]
&=&
\sum_{k \in \zzz} (-1)^k \, 
q^{(m+\frac12)k^2} e^{-2\pi imk(z_1-z_3)}
\underbrace{\Phi^{(-)[\frac12, \frac12]}_1
(\tau, \,\ z_1, \,\ -z_1+2k\tau, \,\ 0)}_{
\substack{|| \\[-2mm] {\displaystyle 
-i \, e^{-2\pi ikz_1}\frac{\eta(\tau)^3}{\vartheta_{11}(\tau, z_1)}
}}}
\nonumber
\\[1mm]
& &
- \,\ \sum_{k \in \zzz} (-1)^k \, 
q^{(m+\frac12)k^2}
e^{2\pi imk(z_1-z_3)}
\underbrace{\Phi^{(-)[\frac12, \frac12]}_1
(\tau, \,\ z_3, \,\ -z_3+2k\tau, \,\ 0)}_{
\substack{|| \\[-2mm] {\displaystyle 
-i \, e^{-2\pi ikz_3}\frac{\eta(\tau)^3}{\vartheta_{11}(\tau, z_3)}
}}}
\\[2mm]
&=& - \, i \, 
\frac{\eta(\tau)^3}{\vartheta_{11}(\tau, z_1)}
\underbrace{\sum_{k \in \zzz} (-1)^k \, 
q^{(m+\frac12)k^2}e^{2\pi i(m+\frac12)k \cdot 
\frac{-m(z_1-z_3)-z_1}{m+\frac12}} }_{
\substack{|| \\[0mm] {\displaystyle 
\theta_{0, m+\frac12}^{(-)}\Big(\tau, \, 
\frac{-m(z_1-z_3)-z_1}{m+\frac12}\Big)
}}}
\\[1mm]
& &+ \, i \, 
\frac{\eta(\tau)^3}{\vartheta_{11}(\tau, z_3)}
\underbrace{\sum_{k \in \zzz} (-1)^k \, 
q^{(m+\frac12)k^2}
e^{2\pi i(m+\frac12)k \cdot 
\frac{m(z_1-z_3)-z_3}{m+\frac12}} }_{\substack{|| \\[0mm] {\displaystyle 
\theta_{0, m+\frac12}^{(-)}\Big(\tau, \, 
\frac{m(z_1-z_3)-z_3}{m+\frac12}\Big)
}}}
\end{eqnarray*}}
Thus we have 
{\allowdisplaybreaks
\begin{eqnarray*}
& & 
\sum_{j \in \zzz}(-1)^j
q^{(m+\frac12)j^2}
e^{2\pi imj(z_1-z_3)}
\Phi^{[m,\frac12]}(\tau, \,\ z_1, \,\ -z_3+2j\tau, \,\ 0)
\\[0mm]
&=&
- \, i \, \frac{\eta(\tau)^3}{\vartheta_{11}(\tau, z_1)} \, 
\theta_{0, \, m+\frac12}^{(-)}\Big(\tau, \, 
\frac{m(z_1-z_3)+z_1}{m+\frac12}\Big)
\, + \, i \, 
\frac{\eta(\tau)^3}{\vartheta_{11}(\tau, z_3)} \, 
\theta_{0, \, m+\frac12}^{(-)}\Big(\tau, \, 
\frac{m(z_1-z_3)-z_3}{m+\frac12}\Big)
\end{eqnarray*}}
Changing the notation $-z_3 \rightarrow z_2$, we obtain the formula 
\eqref{n3:eqn:2022-824e}, proving Lemma \ref{n3:lemma:2022-824b}.
\end{proof}

\vspace{1mm}

\begin{prop}
\label{n3:prop:2022-824a}
For $m \in \frac12 \nnn$, the following formula holds:
{\allowdisplaybreaks
\begin{eqnarray}
& & \hspace{-8mm}
\theta_{0, \, m+\frac12}^{(-)}
\Big(\tau, \, \frac{m(z_1-z_2)}{m+\frac12}\Big) \, 
\Phi^{[m,\frac12]}(\tau, \, z_1, \, z_2, \, 0)
\nonumber
\\[2mm]
&=&
\bigg[
\sum_{\substack{(j, k) \, \in \, \zzz \times \frac12 \zzz_{\rm odd} \\[1mm]
0 \, < \, k \, < \, 2mj
}}
-
\sum_{\substack{(j, k) \, \in \, \zzz \times \frac12 \zzz_{\rm odd} \\[1mm]
2mj \, < \, k \, < \, 0
}}
\bigg] \, 
(-1)^j \, q^{(m+\frac12)j^2}
e^{2\pi imj(z_2-z_1)}
\nonumber
\\[1mm]
& & \hspace{30mm}
\times \,\ 
e^{\pi ik(z_1-z_2)} q^{-\frac{k^2}{4m}} \, 
\big[\theta_{k,m}-\theta_{-k,m}\big](\tau, z_1+z_2) \hspace{30mm}
\nonumber
\\[3mm]
& & \hspace{-7mm}
- \,\ i \, \eta(\tau)^3 \, \Bigg\{
\frac{\displaystyle 
\theta_{0, \, m+\frac12}^{(-)}\Big(\tau, \, 
z_1+z_2+\frac{z_1-z_2}{2m+1}\Big)}{
\vartheta_{11}(\tau, z_1)}
\, + \, 
\frac{\displaystyle 
\theta_{0, \, m+\frac12}^{(-)}\Big(\tau, \, 
z_1+z_2+\frac{z_2-z_1}{2m+1}\Big)
}{\vartheta_{11}(\tau, z_2)} \Bigg\}
\label{n3:eqn:2022-824g}
\end{eqnarray}}
\end{prop}

\begin{proof} We compute the LHS of the equation \eqref{n3:eqn:2022-824e}:
$$
\text{LHS of \eqref{n3:eqn:2022-824e}} \,\ = \,\ 
{\rm (I)}_+ \,\ + \,\ {\rm (I)}_-
$$
where
{\allowdisplaybreaks
\begin{eqnarray*}
{\rm (I)}_+ &:=&
\sum_{j \in \zzz_{\geq 0}} (-1)^j
q^{(m+\frac12)j^2} e^{2\pi imj(z_1+z_2)} \, 
\Phi^{[m,\frac12]}(\tau, \, z_1, \, z_2+2j\tau, \, 0)
\\[1mm]
{\rm (I)}_- &:=&
\sum_{j \in \zzz_{< 0}} (-1)^j
q^{(m+\frac12)j^2} e^{2\pi imj(z_1+z_2)} \, 
\Phi^{[m,\frac12]}(\tau, \, z_1, \, z_2+2j\tau, \, 0)
\end{eqnarray*}}
These functions ${\rm (I)}_{\pm}$ are computed by using Lemma 2.5 
in \cite{W2022b} as follows:
{\allowdisplaybreaks
\begin{eqnarray*}
& & \hspace{-10mm}
{\rm (I)}_+ \,\ = \,\ 
\sum_{j \in \zzz_{\geq 0}}(-1)^j
q^{(m+\frac12)j^2}
e^{2\pi im(j-\frac{1}{4(m+\frac12)})(z_1+z_2)}e^{-4\pi imjz_1}
\\[0mm]
& & \hspace{-5mm}
\times \bigg\{
\Phi^{[m,\frac12]}(\tau, z_1, z_2, 0)
- \hspace{-2mm}
\sum_{k=0}^{2mj-1} \hspace{-2mm}
e^{\pi i(k+\frac12)(z_1-z_2)} q^{-\frac{(k+\frac12)^2}{4m}} \, 
\big[\theta_{k+\frac12,m}-\theta_{-(k+\frac12),m}\big](\tau, z_1+z_2)
\bigg\}
\\[1mm]
&=&
\sum_{j \in \zzz_{\geq 0}}(-1)^j
q^{(m+\frac12)j^2}
e^{2\pi imj(z_2-z_1)}
\Phi^{[m,\frac12]}(\tau,  z_1, z_2,  0)
\\[1mm]
& & \hspace{-5mm}
- \,\sum_{j \in \zzz_{>0}} \sum_{k=0}^{2mj-1} \, (-1)^j
q^{(m+\frac12)j^2}
e^{2\pi imj(z_2-z_1)}
e^{\pi i(k+\frac12)(z_1-z_2)} q^{-\frac{(k+\frac12)^2}{4m}} \, 
\\[-2mm]
& & \hspace{30mm}
\times \,\ 
\big[\theta_{k+\frac12,m}-\theta_{-(k+\frac12),m}\big](\tau, z_1+z_2)
\\[3mm]
& & \hspace{-10mm}
{\rm (I)}_- \,\ = \,\ 
\sum_{j \in \zzz_{< 0}}(-1)^j
q^{(m+\frac12)(j-\frac{1}{4(m+\frac12)})^2}
e^{2\pi imj(z_1+z_2)}e^{-4\pi imjz_1}
\\[1mm]
& & \hspace{-5mm}
\times \bigg\{
\Phi^{[m,\frac12]}(\tau, z_1, z_2, 0)
+
\sum_{\substack{k \in \zzz \\[1mm] 2mj \leq k <0}}
e^{\pi i(k+\frac12)(z_1-z_2)} q^{-\frac{(k+\frac12)^2}{4m}} \, 
\big[\theta_{k+\frac12,m}-\theta_{-(k+\frac12),m}\big](\tau, z_1+z_2)
\bigg\}
\\[3mm]
&=&
\sum_{j \in \zzz_{< 0}}(-1)^j
q^{(m+\frac12)j^2}
e^{2\pi imj(z_2-z_1)}
\Phi^{[m,\frac12]}(\tau, z_1, z_2, 0)
\\[0mm]
& & \hspace{-5mm}
+ \,\ \sum_{j \in \zzz_{< 0}} 
\sum_{\substack{k \in \zzz \\[1mm] 2mj \, \leq \, k \, < \, 0}} 
(-1)^j \, q^{(m+\frac12)j^2}
e^{2\pi imj(z_2-z_1)}
e^{\pi i(k+\frac12)(z_1-z_2)} q^{-\frac{(k+\frac12)^2}{4m}}
\\[-4mm]
& & \hspace{40mm}
\times \,\ 
\big[\theta_{k+\frac12,m}-\theta_{-(k+\frac12),m}\big](\tau, z_1+z_2)
\end{eqnarray*}}
Then we have 
{\allowdisplaybreaks
\begin{eqnarray}
& & \hspace{-10mm}
{\rm (I)}_+ + {\rm (I)}_- 
\,\ = \,\ 
\underbrace{\sum_{j\in \zzz}(-1)^j
q^{(m+\frac12)j^2} 
e^{2\pi i(m+\frac12)j \cdot \frac{m(z_2-z_1)}{m+\frac12}}
}_{\substack{|| \\[-1.5mm] {\displaystyle \hspace{5mm}
\theta_{0, \, m+\frac12}^{(-)}\Big(\tau, \, \frac{m(z_2-z_1)}{m+\frac12}\Big)
}}}
\Phi^{[m,\frac12]}(\tau, z_1, z_2, 0)
\nonumber
\\[1mm]
& &- \,\, 
\Big[\sum_{j \in \zzz_{> 0}} \, 
\sum_{\substack{k \in \zzz \\[1mm] 0 \leq \, k \, < \, 2mj}}
-
\sum_{j \in \zzz_{< 0}} \, 
\sum_{\substack{k \in \zzz \\[1mm] 2mj \, \leq \, k \, < \, 0}} \Big] \, 
(-1)^j \, q^{(m+\frac12)j^2}
e^{2\pi imj(z_2-z_1)}
\nonumber
\\[0mm]
& & \hspace{30mm}
\times \,\ 
e^{\pi i(k+\frac12)(z_1-z_2)} q^{-\frac{(k+\frac12)^2}{4m}} \, 
\big[\theta_{k+\frac12,m}-\theta_{-(k+\frac12),m}\big](\tau, z_1+z_2)
\nonumber
\\[1mm]
&=&
\theta_{0, \, m+\frac12}^{(-)}
\Big(\tau, \, \frac{m(z_1-z_2)}{m+\frac12}\Big) \, 
\Phi^{[m,\frac12]}(\tau, z_1, z_2, 0)
\nonumber
\\[2mm]
& & \hspace{-5mm}
- \, \bigg[
\sum_{\substack{(j, k) \, \in \, \zzz \times \frac12 \zzz_{\rm odd} \\[1mm]
0 \, < \, k \, < \, 2mj
}}
-
\sum_{\substack{(j, k) \, \in \, \zzz \times \frac12 \zzz_{\rm odd} \\[1mm]
2mj \, < \, k \, < \, 0
}}
\bigg] \, 
(-1)^j \, q^{(m+\frac12)j^2}
e^{2\pi imj(z_2-z_1)}
\nonumber
\\[1mm]
& & \hspace{30mm}
\times \,\ 
e^{\pi ik(z_1-z_2)} q^{-\frac{k^2}{4m}} \, 
\big[\theta_{k,m}-\theta_{-k,m}\big](\tau, z_1+z_2)
\label{n3:eqn:2022-824h}
\end{eqnarray}}
Then by \eqref{n3:eqn:2022-824e} and \eqref{n3:eqn:2022-824h}
we obtain \eqref{n3:eqn:2022-824g}, proving Proposition
\ref{n3:prop:2022-824a}.
\end{proof}

\section{Formula for $\Phi^{[\frac{m}{2},s]}(2\tau, 
z+\frac{\tau}{2}-\frac12, \, 
z-\frac{\tau}{2}+\frac12, \, \frac{\tau}{8})$}
\label{sec:formula:Phi}

\medskip

\begin{lemma} \quad 
\label{n3:lemma:2022-825a}
For $m \in \nnn$ and $p \in \zzz$, the following formulas hold:
\begin{subequations}
\begin{enumerate}
\item[{\rm 1)}] \,\ $q^{-\frac{m}{16}(2p+1)^2} \, 
\theta_{1+m(p+\frac12), m+1}^{(\sigma(m))}(\tau,0) \, 
\Phi^{[\frac{m}{2},0]}\Big(2\tau, \,\ 
z+\dfrac{\tau}{2}-\dfrac12+p\tau, \,\ 
z-\dfrac{\tau}{2}+\dfrac12-p\tau, \,\ 0\Big)$
{\allowdisplaybreaks
\begin{eqnarray}
&=& 
\bigg[
\sum_{\substack{j, \, k \, \in \zzz \\[1mm] 0 \, < \, k \, \leq \, mj}}
-
\sum_{\substack{j, \, k \, \in \zzz \\[1mm] mj \, < \, k \, \leq \, 0}} 
\bigg] (-1)^{(m+1)j+k} \, 
q^{(m+1)(j+\frac{1}{2(m+1)}+\frac{m(2p+1)}{4(m+1)})^2
\, - \, \frac{1}{m}(k+\frac{m(2p+1)}{4})^2} 
\nonumber
\\[-2mm]
& & \hspace{60mm}
\times \,\ 
\big[\theta_{2k,m}-\theta_{-2k,m}\big](\tau, z)
\nonumber
\\[1mm]
& & 
+ \,\ \eta(2\tau)^3 \, \bigg\{
\frac{\theta_{p-\frac12, m+1}(\tau,z)}{\theta_{p-\frac12, 1}(\tau,z)}
\, - \, 
\frac{\theta_{-p+\frac12, m+1}(\tau,z)}{\theta_{-p+\frac12, 1}(\tau,z)}
\bigg\}
\label{n3:eqn:2022-825a}
\end{eqnarray}}
\item[{\rm 2)}] \,\ $q^{-\frac{m}{16}(2p+1)^2} \, 
\theta_{m(p+\frac12), m+1}^{(\sigma(m))}(\tau,0) \, 
\Phi^{[\frac{m}{2},\frac12]}\Big(2\tau, \,\ 
z+\dfrac{\tau}{2}-\dfrac12+p\tau, \,\ 
z-\dfrac{\tau}{2}+\dfrac12-p\tau, \,\ 0\Big)$
{\allowdisplaybreaks
\begin{eqnarray}
&=& 
\bigg[
\sum_{\substack{(j, k) \, \in \, \zzz \times \frac12 \zzz_{\rm odd} \\[1mm]
0 \, < \, k \, < \, mj}}
-
\sum_{\substack{(j, k) \, \in \, \zzz \times \frac12 \zzz_{\rm odd} \\[1mm]
mj \, < \, k \, < \, 0}}
\bigg] \, 
(-1)^{(m+1)j+k} \, q^{(m+1)(j+\frac{m(2p+1)}{4(m+1)})^2
\, - \, \frac{1}{m}(k+\frac{m(2p+1)}{4})^2} 
\nonumber
\\[-6mm]
& & \hspace{65mm}
\times \,\ \big[\theta_{2k,m}-\theta_{-2k,m}\big](\tau, z)
\nonumber
\\[2mm]
& & 
- \,\ i \,\ \eta(2\tau)^3 \, \bigg\{
\frac{\theta_{p+\frac12, m+1}(\tau,z)}{\theta_{p-\frac12, 1}(\tau,z)}
\, - \, 
\frac{\theta_{-p-\frac12, m+1}(\tau,z)}{\theta_{-p+\frac12, 1}(\tau,z)}
\bigg\}
\label{n3:eqn:2022-827a}
\end{eqnarray}}
\end{enumerate}
\end{subequations}
where $\sigma(m)$ is defined by \eqref{n3:eqn:2022-825b}.
\end{lemma}

\begin{proof} Letting $(m, \tau) \rightarrow (\frac{m}{2}, 2\tau)$ 
in \eqref{n3:eqn:2022-823p} and \eqref{n3:eqn:2022-824g} and using 
$\theta_{\frac{j}{2}, \frac{m}{2}}^{(\pm)}(2\tau, 2z) 
= \theta_{j,m}^{(\pm)}(\tau,z)$, we have 
\begin{subequations}
{\allowdisplaybreaks
\begin{eqnarray}
& & \hspace{-10mm}
\theta_{1, m+1}^{(-)}\Big(\tau, \frac{m(z_1-z_2)}{2(m+1)}\Big) \, 
\Phi^{[\frac{m}{2},0]}(2\tau, \, z_1, \, z_2, \, 0)
\nonumber
\\[2mm]
& = & \bigg[
\sum_{\substack{j, \, k \, \in \zzz \\[1mm] 0 \, < \, k \, \leq \, mj}}
-
\sum_{\substack{j, \, k \, \in \zzz \\[1mm] mj \, < \, k \, \leq \, 0}} 
\bigg] \, 
(-1)^j \, q^{2(\frac{m}{2}+\frac12)(j+\frac{1}{4(\frac{m}{2}+\frac12)})^2}
e^{\pi im(j+\frac{1}{4(\frac{m}{2}+\frac12)})(z_1-z_2)}
\nonumber
\\[-7mm]
& & \hspace{50mm}
\times \,\ 
e^{-\pi ik(z_1-z_2)} q^{-\frac{k^2}{m}} \, 
\big[\theta_{2k,m}-\theta_{-2k,m}\big]\Big(\tau, \frac{z_1+z_2}{2}\Big)
\nonumber
\\[2mm]
& & \hspace{-7mm}
- \, i \, \eta(2\tau)^3 \Bigg\{
\frac{\displaystyle 
\theta_{-1, m+1}^{(-)}\Big(\tau, 
\frac{z_1+z_2}{2}+\frac{z_1-z_2}{2(m+1)}\Big)}{
\vartheta_{11}(2\tau, z_1)}
+ 
\frac{\displaystyle 
\theta_{1, m+1}^{(-)}\Big(\tau, 
\frac{z_1+z_2}{2}+\frac{z_2-z_1}{2(m+1)}\Big)
}{\vartheta_{11}(2\tau, z_2)} \Bigg\}
\label{n3:eqn:2022-825c}
\\[3mm]
& & \hspace{-10mm}
\theta_{0, m+1}^{(-)}\Big(\tau, \frac{m(z_1-z_2)}{2(m+1)}\Big) \, 
\Phi^{[\frac{m}{2},\frac12]}\Big(2\tau, \, 
z+\frac{\tau}{2}-\frac12+p\tau, \, 
z-\frac{\tau}{2}+\frac12-p\tau, \, 0\Big)
\nonumber
\\[2mm]
&=&
\bigg[
\sum_{\substack{(j, k) \, \in \, \zzz \times \frac12 \zzz_{\rm odd} \\[1mm]
0 \, < \, k \, < \, mj}}
-
\sum_{\substack{(j, k) \, \in \, \zzz \times \frac12 \zzz_{\rm odd} \\[1mm]
mj \, < \, k \, < \, 0}}
\bigg] \, 
(-1)^j \, q^{2(\frac{m}{2}+\frac12)j^2}
e^{\pi imj(z_1-z_2)}
\nonumber
\\[-5mm]
& & \hspace{55mm}
\times \,\ 
e^{-\pi ik(z_1-z_2)} q^{-\frac{k^2}{m}} \, 
\big[\theta_{2k,m}-\theta_{-2k,m}\big]\Big(\tau, \frac{z_1+z_2}{2}\Big)
\nonumber
\\[2mm]
& & \hspace{-9mm}
- \, i \, \eta(2\tau)^3 \Bigg\{
\frac{\displaystyle 
\theta_{0, m+1}^{(-)}\Big(2\tau,  
\frac{z_1+z_2}{2}+\frac{z_1-z_2}{2(m+1)}\Big)}{
\vartheta_{11}(2\tau, z_1)}
+ 
\frac{\displaystyle 
\theta_{0, m+1}^{(-)}\Big(2\tau, 
\frac{z_1+z_2}{2}+\frac{z_2-z_1}{2(m+1)}\Big)
}{\vartheta_{11}(2\tau, z_2)} \Bigg\}
\label{n3:eqn:2022-827b}
\end{eqnarray}}
\end{subequations}
We put $\left\{
\begin{array}{ccc}
z_1 &=& z+\frac{\tau}{2}-\frac12+p\tau \\[2mm]
z_2 &=& z-\frac{\tau}{2}+\frac12-p\tau
\end{array} \right. $, then 
$\left\{
\begin{array}{ccc}
z_1-z_2 &=& (2p+1)\tau-1 \\[0mm]
z_1+z_2 &=& 2 z
\end{array}\right. $ and the above equations \eqref{n3:eqn:2022-825c}
and \eqref{n3:eqn:2022-827b} become as follows:
%
\begin{subequations}
{\allowdisplaybreaks
\begin{eqnarray}
& &\hspace{-10mm}
\theta_{1, m+1}^{(-)}\Big(\tau, \frac{m((2p+1)\tau-1)}{2(m+1)}\Big)  \, 
\Phi^{[\frac{m}{2},0]}\Big(2\tau, \, 
z+\frac{\tau}{2}-\frac12+p\tau, \, 
z-\frac{\tau}{2}+\frac12-p\tau, \, 0\Big)
\nonumber
\\[1mm]
&=&e^{-\frac{\pi im}{2(m+1)}}
\bigg[
\sum_{\substack{j, \, k \, \in \zzz \\[1mm] 0 \, < \, k \, \leq \, mj}}
-
\sum_{\substack{j, \, k \, \in \zzz \\[1mm] mj \, < \, k \, \leq \, 0}} 
\bigg] \, 
(-1)^{(m+1)j} \, q^{(m+1)(j+\frac{1}{2(m+1)})^2} \, 
q^{\frac{m}{2}(2p+1)(j+\frac{1}{2(m+1)})} \, 
\nonumber
\\[1mm]
& & \hspace{30mm}
\times \,\ 
(-1)^k q^{-\frac{k}{2}(2p+1)} \, 
q^{-\frac{k^2}{m}} \, 
\big[\theta_{2k,m}-\theta_{-2k,m}\big](\tau, z)
\nonumber
\\[3mm]
& & \hspace{-5mm}
- \, i \, \eta(2\tau)^3 \, \Bigg\{
\underbrace{\frac{\displaystyle 
\theta_{-1, \, m+1}^{(-)}\Big(\tau, \, z+\frac{(2p+1)\tau-1}{2(m+1)}\Big)
}{\vartheta_{11}(2\tau, z_1)}
+ 
\frac{\displaystyle 
\theta_{1, \, m+1}^{(-)}\Big(\tau, \, z+\frac{-(2p+!)\tau+1}{2(m+1)}\Big)
}{\vartheta_{11}(2\tau, z_2)}}_{\rm (I)}
\Bigg\}
\label{n3:eqn:2022-825d}
\\[3mm]
& & \hspace{-10mm}
\theta_{0, m+1}^{(-)}\Big(\tau, 
\frac{m((2p+1)\tau-1)}{2(m+1)}
\Big) \, 
\Phi^{[\frac{m}{2},\frac12]}\Big(2\tau, \, 
z+\frac{\tau}{2}-\frac12+p\tau, \, 
z-\frac{\tau}{2}+\frac12-p\tau, \, 0\Big)
\nonumber
\\[3mm]
&=&
\bigg[
\sum_{\substack{(j, k) \, \in \, \zzz \times \frac12 \zzz_{\rm odd} \\[1mm]
0 \, < \, k \, < \, mj
}}
-
\sum_{\substack{(j, k) \, \in \, \zzz \times \frac12 \zzz_{\rm odd} \\[1mm]
mj \, < \, k \, < \, 0
}}
\bigg] \, 
(-1)^{(m+1)j} \, q^{(m+1)j^2} \, 
q^{\frac{m}{2}(2p+1)j} 
\nonumber
\\[2mm]
& & \hspace{30mm}
\times \,\ 
(-1)^k q^{-\frac{k}{2}(2p+1)} \, 
q^{-\frac{k^2}{m}} \, 
\big[\theta_{2k,m}-\theta_{-2k,m}\big](\tau, z)
\nonumber
\\[3mm]
& & \hspace{-7mm}
- \, i \, \eta(2\tau)^3 \Bigg\{
\underbrace{\frac{\displaystyle 
\theta_{0, \, m+1}^{(-)}\Big(\tau, 
\frac{z_1+z_2}{2}+\frac{z_1-z_2}{2(m+1)}\Big)}{
\vartheta_{11}(2\tau, z_1)}
+ 
\frac{\displaystyle 
\theta_{0, \, m+1}^{(-)}\Big(\tau, 
\frac{z_1+z_2}{2}+\frac{z_2-z_1}{2(m+1)}\Big)
}{\vartheta_{11}(2\tau, z_2)} }_{\rm (II)}
\Bigg\}
\label{n3:eqn:2022-827c}
\end{eqnarray}}
\end{subequations}
(I) and (II) and the LHS's of these equations \eqref{n3:eqn:2022-825d} 
and \eqref{n3:eqn:2022-827c} are computed by using 
Notes \ref{n3:note:2022-825a}, \ref{n3:note:2022-825b} and 
\ref{n3:note:2022-825c} as follows:
\begin{subequations}
{\allowdisplaybreaks
\begin{eqnarray}
& & \hspace{-10mm}
{\rm (I)} =
\frac{\displaystyle 
e^{\frac{\pi i}{2(m+1)}} q^{-\frac{(2p+1)^2}{16(m+1)}} 
e^{-\pi i(p+\frac12)z} 
\theta_{p-\frac12, m+1}(\tau,z)
}{\displaystyle 
q^{-\frac{1}{16}(2p+1)^2} \, e^{-\pi i(p+\frac12)z} \, 
\theta_{p-\frac12, 1}(\tau,z)}
+ 
\frac{\displaystyle 
e^{\frac{\pi i}{2(m+1)}} q^{-\frac{(2p+1)^2}{16(m+1)}} 
e^{\pi i(p+\frac12)z} 
\theta_{-p+\frac12, m+1}(\tau,z)
}{\displaystyle - \, 
q^{-\frac{1}{16}(2p+1)^2} \, e^{\pi i(p+\frac12)z} \, 
\theta_{-p+\frac12, 1}(\tau,z)}
\nonumber
\\[1mm]
&=&
e^{\frac{\pi i}{2(m+1)}} \, 
q^{\frac{m}{16(m+1)}(2p+1)^2} \, \bigg\{
\frac{\theta_{p-\frac12, m+1}(\tau,z)}{\theta_{p-\frac12, 1}(\tau,z)}
\, - \, 
\frac{\theta_{-p+\frac12, m+1}(\tau,z)}{\theta_{-p+\frac12, 1}(\tau,z)}
\bigg\}
\label{n3:eqn:2022-826b1}
\\[2mm]
& & \hspace{-10mm}
\text{LHS of \eqref{n3:eqn:2022-825d}}
\,\ = \,\ 
e^{-\frac{\pi im}{2(m+1)}} \, 
q^{-\frac{m^2}{16(m+1)}(2p+1)^2} \, 
\theta_{1+m(p+\frac12), m+1}^{(\sigma(m))}(\tau,0)
\nonumber
\\[0mm]
& & \hspace{20mm}
\times \,\ 
\Phi^{[\frac{m}{2},0]}\Big(2\tau, \,\ 
z+\frac{\tau}{2}-\frac12+p\tau, \,\ 
z-\frac{\tau}{2}+\frac12-p\tau, \,\ 0\Big)
\label{n3:eqn:2022-826b2}
\end{eqnarray}}
\end{subequations}
and
\begin{subequations}
{\allowdisplaybreaks
\begin{eqnarray}
& & \hspace{-10mm}
{\rm (II)} \, = \, 
\frac{\displaystyle 
q^{-\frac{(2p+1)^2}{16(m+1)}} \, 
e^{-\pi i(p+\frac12)z} \, 
\theta_{p+\frac12, m+1}(\tau,z)
}{\displaystyle 
q^{-\frac{1}{16}(2p+1)^2} \, e^{-\pi i(p+\frac12)z} \, 
\theta_{p-\frac12, 1}(\tau,z)}
\, + \, 
\frac{\displaystyle 
q^{-\frac{(2p+1)^2}{16(m+1)}} \, 
e^{\pi i(p+\frac12)z} \, 
\theta_{-p-\frac12, m+1}(\tau,z)
}{\displaystyle - \, 
q^{-\frac{1}{16}(2p+1)^2} \, e^{\pi i(p+\frac12)z} \, 
\theta_{-p+\frac12, 1}(\tau,z)}
\nonumber
\\[2mm]
&=&
q^{\frac{m}{16(m+1)}(2p+1)^2} \, \bigg\{
\frac{\theta_{p+\frac12, m+1}(\tau,z)}{\theta_{p-\frac12, 1}(\tau,z)}
\, - \, 
\frac{\theta_{-p-\frac12, m+1}(\tau,z)}{\theta_{-p+\frac12, 1}(\tau,z)}
\bigg\}
\label{n3:eqn:2022-827d1}
\\[2mm]
& & \hspace{-10mm}
\text{LHS of \eqref{n3:eqn:2022-827c}}
\,\ = \,\ 
q^{-\frac{m^2}{16(m+1)}(2p+1)^2} \, 
\theta_{m(p+\frac12), m+1}^{(\sigma(m))}(\tau,0)
\nonumber
\\[1mm]
& & \hspace{20mm}
\times \,\ 
\Phi^{[\frac{m}{2},\frac12]}\Big(2\tau, \,\ 
z+\frac{\tau}{2}-\frac12+p\tau, \,\ 
z-\frac{\tau}{2}+\frac12-p\tau, \,\ 0\Big)
\label{n3:eqn:2022-827d2}
\end{eqnarray}}
\end{subequations}
Substituting \eqref{n3:eqn:2022-826b1} and 
\eqref{n3:eqn:2022-826b2} into \eqref{n3:eqn:2022-825d}, we have
{\allowdisplaybreaks
\begin{eqnarray}
& & \hspace{-8mm}
e^{-\frac{\pi im}{2(m+1)}} 
q^{-\frac{m^2}{16(m+1)}(2p+1)^2} 
\theta_{1+m(p+\frac12), m+1}^{(\sigma(m))}(\tau,0)
\Phi^{[\frac{m}{2},0]}\Big(2\tau, \, 
z+\frac{\tau}{2}-\frac12+p\tau, \, 
z-\frac{\tau}{2}+\frac12-p\tau, \, 0\Big)
\nonumber
\\[2mm]
&=&
e^{-\frac{\pi im}{2(m+1)}}\Big[
\sum_{\substack{j, \, k \, \in \zzz \\[1mm] 0 \, < \, k \, \leq \, mj}}
-
\sum_{\substack{j, \, k \, \in \zzz \\[1mm] mj \, < \, k \, \leq \, 0}} 
\Big] \, 
(-1)^{(m+1)j} \, q^{(m+1)(j+\frac{1}{2(m+1)})^2} \, 
q^{\frac{m}{2}(2p+1)(j+\frac{1}{2(m+1)})} 
\nonumber
\\[0mm]
& & \hspace{40mm}
\times \,\ 
(-1)^k q^{-\frac{k^2}{m}-\frac{k}{2}(2p+1)} \, 
\big[\theta_{2k,m}-\theta_{-2k,m}\big](\tau, z)
\nonumber
\\[2mm]
& & \hspace{-2mm}
- \,\ i \, \eta(2\tau)^3 \, 
e^{\frac{\pi i}{2(m+1)}} \, 
q^{\frac{m}{16(m+1)}(2p+1)^2} \, \Bigg\{
\frac{\theta_{p-\frac12, m+1}(\tau,z)}{\theta_{p-\frac12, 1}(\tau,z)}
\, - \, 
\frac{\theta_{-p+\frac12, m+1}(\tau,z)}{\theta_{-p+\frac12, 1}(\tau,z)}
\Bigg\}
\label{n3:eqn:2022-826c}
\end{eqnarray}}
And substituting \eqref{n3:eqn:2022-827d1} and 
\eqref{n3:eqn:2022-827d2} into \eqref{n3:eqn:2022-827c}, we have
{\allowdisplaybreaks
\begin{eqnarray}
& & \hspace{-7mm}
q^{-\frac{m^2}{16(m+1)}(2p+1)^2} \, 
\theta_{m(p+\frac12), m+1}^{(\sigma(m))}(\tau,0) \, 
\Phi^{[\frac{m}{2},\frac12]}\Big(2\tau, \,\ 
z+\frac{\tau}{2}-\frac12+p\tau, \,\ 
z-\frac{\tau}{2}+\frac12-p\tau, \,\ 0\Big)
\nonumber
\\[2mm]
&=&
\bigg[
\sum_{\substack{(j, k) \, \in \, \zzz \times \frac12 \zzz_{\rm odd} \\[1mm]
0 \, < \, k \, < \, mj}}
-
\sum_{\substack{(j, k) \, \in \, \zzz \times \frac12 \zzz_{\rm odd} \\[1mm]
mj \, < \, k \, < \, 0}}
\bigg] \, 
(-1)^{(m+1)j} \, q^{(m+1)j^2} \, q^{\frac{m}{2}(2p+1)j} 
\nonumber
\\[-8mm]
& & \hspace{60mm}
\times \,\ 
(-1)^k q^{-\frac{k^2}{m}-\frac{k}{2}(2p+1)} \, 
\big[\theta_{2k,m}-\theta_{-2k,m}\big](\tau, z)
\nonumber
\\[2mm]
& & \hspace{-2mm}
- \,\ i \, \eta(2\tau)^3 \, 
q^{\frac{m}{16(m+1)}(2p+1)^2} \, \bigg\{
\frac{\theta_{p+\frac12, m+1}(\tau,z)}{\theta_{p-\frac12, 1}(\tau,z)}
\, - \, 
\frac{\theta_{-p-\frac12, m+1}(\tau,z)}{\theta_{-p+\frac12, 1}(\tau,z)}
\bigg\}
\label{n3:eqn:2022-827e}
\end{eqnarray}}
Multiplying \, 
$e^{\frac{\pi im}{2(m+1)}} \, q^{-\frac{m}{16(m+1)}(2p+1)^2}$ 
to both sides of \eqref{n3:eqn:2022-826c} we obtain the formula 
\eqref{n3:eqn:2022-825a},
and multiplying \, $q^{-\frac{m}{16(m+1)}(2p+1)^2}$ 
to both sides of \eqref{n3:eqn:2022-827e} we obtain the formula 
\eqref{n3:eqn:2022-827a}. Thus the proof of Lemma 
\ref{n3:lemma:2022-825a} is completed.
\end{proof}

\vspace{1mm}

\begin{lemma} 
\label{n3:lemma:2022-827b}
For $m \in \nnn, \, s \in \frac12 \zzz$ and $a \in \zzz_{\geq 0}$ 
the following formula holds:
{\allowdisplaybreaks
\begin{eqnarray}
& & \hspace{-7mm}
\Phi^{(\pm)[\frac{m}{2},s]}(2\tau, \, z_1+2a\tau, \, z_2-2a\tau, \, 0) 
\nonumber
\\[1mm]
&=&
(\pm 1)^a \, e^{\pi ima(z_1-z_2)} \, q^{ma^2} \, \bigg\{
\Phi^{(\pm)[\frac{m}{2},s]}(2\tau, z_1,z_2,0)
\nonumber
\\[1mm]
& & \hspace{-5mm}
- \sum_{\substack{k \in \zzz \\[1mm] 1 \leq k \leq am}}
e^{-\pi i(k-s)(z_1-z_2)} \, 
q^{- \frac{(k-s)^2}{m}} \, 
\big[\theta^{(\pm)}_{2(k-s), \, m} - 
\theta^{(\pm)}_{-2(k-s), \, m}\big] \Big(\tau, \frac{z_1+z_2}{2}\Big)
\bigg\}
\label{n3:eqn:2022-827f}
\end{eqnarray}}
\end{lemma}

\begin{proof} This formula follows immediately from Lemma 4.3 
in \cite{W2022c} by letting $(m, \tau) \rightarrow (\frac{m}{2}, 2\tau)$.
\end{proof}

\vspace{1mm}

\begin{lemma} 
\label{n3:lemma:2022-827c}
For $m \in \nnn, \, s \in \frac12 \zzz$ and $a \in \zzz_{\geq 0}$ 
the following formula holds:
{\allowdisplaybreaks
\begin{eqnarray}
& & \hspace{-7mm}
\Phi^{(\pm)[\frac{m}{2},s]}\Big(2\tau, \,\ 
z+\frac{\tau}{2}-\frac12, \,\ 
z-\frac{\tau}{2}+\frac12, \,\ \frac{\tau}{8}\Big)
\nonumber
\\[2mm]
&=&
(\pm 1)^a \, (-1)^{ma} \, q^{-m(a+\frac14)^2} \, 
\Phi^{(\pm)[\frac{m}{2},s]}\Big(2\tau, \,\ 
z+\frac{\tau}{2}-\frac12+2a\tau, \,\ 
z-\frac{\tau}{2}+\frac12-2a\tau, \,\ 0\Big) 
\nonumber
\\[1mm]
& & \hspace{-5mm}
+ \,\ e^{-\pi is}
\sum_{\substack{k \in \zzz \\[1mm] 1 \, \leq \, k \, \leq \, am}} 
(-1)^k  \, q^{-\frac{1}{m}(k-s+\frac{m}{4})^2} \, 
\big[\theta^{(\pm)}_{2(k-s), \, m} - 
\theta^{(\pm)}_{-2(k-s), \, m}\big] (\tau, z)
\label{n3:eqn:2022-827g}
\end{eqnarray}}
\end{lemma}

\begin{proof} 
Letting $\left\{
\begin{array}{ccc}
z_1 &=& z+\frac{\tau}{2}-\frac12 \\[1mm]
z_2 &=& z-\frac{\tau}{2}+\frac12
\end{array} \right. $ i.e, $\left\{
\begin{array}{ccc}
z_1-z_2 &=& \tau-1 \\[1mm]
z_1+z_2 &=& 2 z
\end{array}\right. $ in the formula \eqref{n3:eqn:2022-827f}, 
we have 

{\allowdisplaybreaks
\begin{eqnarray*}
& & \hspace{-7mm}
\Phi^{(\pm)[\frac{m}{2},s]}\Big(2\tau, \,\ 
z+\frac{\tau}{2}-\frac12+2a\tau, \,\ 
z-\frac{\tau}{2}+\frac12-2a\tau, \,\ 0\Big) 
\\[2mm]
&=&
(\pm 1)^a \, 
e^{\pi ima(\tau-1)} \, 
q^{ma^2} \, \bigg\{
\Phi^{(\pm)[\frac{m}{2},s]}(2\tau, z_1,z_2,0)
\\[0mm]
& & 
- \sum_{\substack{k \in \zzz \\[1mm] 1 \leq k \leq am}}
e^{-\pi i(k-s)(\tau-1)} \, 
q^{- \frac{(k-s)^2}{m}} \, 
\big[\theta^{(\pm)}_{2(k-s), \, m} - 
\theta^{(\pm)}_{-2(k-s), \, m}\big] (\tau, z)
\bigg\}
\\[3mm]
&=&
(\pm 1)^a \, (-1)^{ma} \, 
q^{m(a+\frac14)^2} \, \bigg\{q^{-\frac{m}{16}} \, 
\Phi^{(\pm)[\frac{m}{2},s]}(2\tau, z_1,z_2,0)
\\[2mm]
& & - \,\ e^{-\pi is}
\sum_{\substack{k \in \zzz \\[1mm] 1 \, \leq \, k \, \leq \, am}} 
(-1)^k  \, q^{-\frac{1}{m}(k-s+\frac{m}{4})^2} \, 
\big[\theta^{(\pm)}_{2(k-s), \, m} - 
\theta^{(\pm)}_{-2(k-s), \, m}\big] (\tau, z)
\bigg\}
\end{eqnarray*}}
Multiplying \, $(\pm 1)^a \, (-1)^{ma} \, q^{-m(a+\frac14)^2}$,
we have
{\allowdisplaybreaks
\begin{eqnarray*}
& & \hspace{-7mm}
(\pm 1)^a \, (-1)^{ma} \, q^{-m(a+\frac14)^2} \, 
\Phi^{(\pm)[\frac{m}{2},s]}\Big(2\tau, \,\ 
z+\frac{\tau}{2}-\frac12+2a\tau, \,\ 
z-\frac{\tau}{2}+\frac12-2a\tau, \,\ 0\Big) 
\\[2mm]
&=&
\underbrace{q^{-\frac{m}{16}} \, 
\Phi^{(\pm)[\frac{m}{2},s]}(2\tau, z_1,z_2,0)}_{
\substack{|| \\[-1.5mm] {\displaystyle \hspace{-5mm}
\Phi^{(\pm)[\frac{m}{2},s]}
\Big(2\tau, \, z_1, \, z_2, \, \frac{\tau}{8}\Big)
}}}
\\[1mm]
& & 
- \,\ e^{-\pi is}
\sum_{\substack{k \in \zzz \\[1mm] 1 \, \leq \, k \, \leq \, am}} 
(-1)^k  \, q^{-\frac{1}{m}(k-s+\frac{m}{4})^2} \, 
\big[\theta^{(\pm)}_{2(k-s), \, m} - 
\theta^{(\pm)}_{-2(k-s), \, m}\big] (\tau, z)
\end{eqnarray*}}
proving Lemma \ref{n3:lemma:2022-827c}.
\end{proof}

\vspace{1mm}


\begin{prop} 
\label{n3:prop:2022-827a}
For $m \in \nnn$ and $p \in \zzz_{\geq 0}$ the following formulas hold:
\begin{subequations}
\begin{enumerate}
\item[{\rm 1)}] \,\ $\theta^{(\sigma(m))}_{2p-1-\frac{m}{2}, m+1}(\tau,0) \, 
\Phi^{[\frac{m}{2},0]}\Big(2\tau, \,\ 
z+\dfrac{\tau}{2}-\dfrac12, \,\ 
z-\dfrac{\tau}{2}+\dfrac12, \,\ \dfrac{\tau}{8}\Big)$
{\allowdisplaybreaks
\begin{eqnarray}
&=&
(-1)^p \, \eta(2\tau)^3 \, \bigg\{
\frac{\theta_{2p-\frac12, m+1}(\tau,z)}{\theta_{-\frac12, 1}(\tau,z)}
\, - \, 
\frac{\theta_{-2p+\frac12, m+1}(\tau,z)}{\theta_{\frac12, 1}(\tau,z)}
\bigg\}
\nonumber
\\[2mm]
& &\hspace{-5mm}
+ \,\ (-1)^p \, \bigg[
\sum_{\substack{j, \, k \, \in \zzz \\[1mm] 0 \, < \, k \, \leq \, mj}}
-
\sum_{\substack{j, \, k \, \in \zzz \\[1mm] mj \, < \, k \, \leq \, 0}} \bigg] 
(-1)^{(m+1)j+k} \, 
q^{(m+1)(j+\frac{1}{2(m+1)}+\frac{m(4p+1)}{4(m+1)})^2
\, - \, \frac{1}{m}(k+\frac{m(4p+1)}{4})^2} 
\nonumber
\\[-2mm]
& & \hspace{60mm}
\times \,\ 
\big[\theta_{2k,m}-\theta_{-2k,m}\big](\tau, z)
\nonumber
\\[2mm]
& & \hspace{-5mm}
+ \,\ 
\theta^{(\sigma(m))}_{2p-1-\frac{m}{2}, m+1}(\tau,0)
\sum_{\substack{k \, \in \zzz \\[1mm] 1 \, \leq \, k \, \leq \, pm}} 
(-1)^k  \, q^{-\frac{1}{m}(k+\frac{m}{4})^2} \, 
\big[\theta_{2k, \, m} - \theta_{-2k, \, m}\big] (\tau, z)
\label{n3:eqn:2022-827h1}
\end{eqnarray}}
\item[{\rm 2)}] \,\ $\theta^{(\sigma(m))}_{2p-\frac{m}{2}, m+1}(\tau,0) \, 
\Phi^{[\frac{m}{2},\frac12]}\Big(2\tau, \,\ 
z+\dfrac{\tau}{2}-\dfrac12, \,\ 
z-\dfrac{\tau}{2}+\dfrac12, \,\ \dfrac{\tau}{8}\Big)$
{\allowdisplaybreaks
\begin{eqnarray}
&=&- \,\ i \, (-1)^p \, \eta(2\tau)^3 \, \bigg\{
\frac{\theta_{2p+\frac12, m+1}(\tau,z)}{\theta_{-\frac12, 1}(\tau,z)}
\, - \, 
\frac{\theta_{-2p-\frac12, m+1}(\tau,z)}{\theta_{\frac12, 1}(\tau,z)}
\bigg\}
\nonumber
\\[2mm]
& & \hspace{-5mm}
+ \, (-1)^p \bigg[
\sum_{\substack{(j, k) \, \in \, \zzz \times \frac12 \zzz_{\rm odd} \\[1mm]
0 \, < \, k \, < \, mj}}
-
\sum_{\substack{(j, k) \, \in \, \zzz \times \frac12 \zzz_{\rm odd} \\[1mm]
mj \, < \, k \, < \, 0}}
\bigg] 
(-1)^{(m+1)j+k} q^{(m+1)(j+\frac{m(4p+1)}{4(m+1)})^2
\, - \, \frac{1}{m}(k+\frac{m(4p+1)}{4})^2} 
\nonumber
\\[-6mm]
& & \hspace{65mm}
\times \,\ \big[\theta_{2k,m}-\theta_{-2k,m}\big](\tau, z)
\nonumber
\\[3mm]
& & \hspace{-5mm}
- \, i \, 
\theta^{(\sigma(m))}_{2p-\frac{m}{2}, m+1}(\tau,0) \hspace{-2mm}
\sum_{\substack{k \in \zzz \\[1mm] 1 \, \leq \, k \, \leq \, pm}} 
(-1)^k q^{-\frac{1}{m}(k-\frac12+\frac{m}{4})^2}
\big[\theta_{2k-1, \, m} - \theta_{-(2k-1), \, m}\big] (\tau, z)
\label{n3:eqn:2022-827h2}
\end{eqnarray}}
\end{enumerate}
\end{subequations}
where $\sigma(m)$ is defined by \eqref{n3:eqn:2022-825b}.
\end{prop}

\begin{proof} Letting $a=p$ in \eqref{n3:eqn:2022-827g}, we have
{\allowdisplaybreaks
\begin{eqnarray}
& & \hspace{-7mm}
\Phi^{[\frac{m}{2},s]}\Big(2\tau, \,\ 
z+\frac{\tau}{2}-\frac12, \,\ 
z-\frac{\tau}{2}+\frac12, \,\ 
\frac{\tau}{8}\Big)
\nonumber
\\[2mm]
&=&
(-1)^{mp} \, q^{-m(p+\frac14)^2} \, 
\Phi^{[\frac{m}{2},s]}\Big(2\tau, \,\ 
z+\frac{\tau}{2}-\frac12+2p\tau, \,\ 
z-\frac{\tau}{2}+\frac12-2p\tau, \,\ 0\Big) 
\nonumber
\\[2mm]
& & \hspace{-5mm}
+ \,\ e^{-\pi is}
\sum_{\substack{k \, \in \zzz \\[1mm] 1 \, \leq \, k \, \leq \, pm}} 
(-1)^k  \, q^{-\frac{1}{m}(k-s+\frac{m}{4})^2} \, 
\big[\theta_{2(k-s), \, m} - \theta_{-2(k-s), \, m}\big] 
(\tau, z)
\label{n3:eqn:2022-827j}
\end{eqnarray}}

\noindent
\underline{Proof of 1)} : \,\ Multiplying \, 
$\theta^{(\sigma(m))}_{1+m(2p+\frac12), m+1}(\tau,0)$ to 
${\rm \eqref{n3:eqn:2022-827j}}_{s=0}$, we have 
{\allowdisplaybreaks
\begin{eqnarray*}
& & \hspace{-10mm}
\underbrace{\theta^{(\sigma(m))}_{1+m(2p+\frac12),m+1}(\tau,0)}_{
\substack{|| \\[0mm] {\displaystyle \hspace{13mm}
(-1)^{p(m+1)} \theta^{(\sigma(m))}_{2p-1-\frac{m}{2}, m+1}(\tau,0)
}}} \hspace{-10.8mm} 
\Phi^{[\frac{m}{2},0]}\Big(2\tau, \,\ 
z+\frac{\tau}{2}-\frac12, \,\ 
z-\frac{\tau}{2}+\frac12, \,\ 
\frac{\tau}{8}\Big)
\\[2mm]
&=&
(-1)^{mp} 
\underbrace{q^{-m(p+\frac14)^2} 
\theta^{(\sigma(m))}_{1+m(2p+\frac12),m+1}(\tau,0) 
\Phi^{[\frac{m}{2},0]}\Big(2\tau, 
z+\frac{\tau}{2}-\frac12+2p\tau, 
z-\frac{\tau}{2}+\frac12-2p\tau, 0\Big)}_{\rm (I)}
\\[1mm]
& & \hspace{-5mm}
+ \,\ 
\theta^{(\sigma(m))}_{1+m(2p+\frac12),m+1}(\tau,0)
\sum_{\substack{k \, \in \zzz \\[1mm] 1 \, \leq \, k \, \leq \, pm}} 
(-1)^k  \, q^{-\frac{1}{m}(k+\frac{m}{4})^2} \, 
\big[\theta_{2k, \, m} - \theta_{-2k, \, m}\big] (\tau, z)
\end{eqnarray*}}
Here (I) is obtained by ${\rm \eqref{n3:eqn:2022-825a}}_{p \rightarrow 2p}$, 
so we have 
{\allowdisplaybreaks
\begin{eqnarray*}
& & 
(-1)^{p(m+1)} \theta^{(\sigma(m))}_{2p-1-\frac{m}{2}, m+1}(\tau,0)
\Phi^{[\frac{m}{2},0]}\Big(2\tau, \,\ 
z+\frac{\tau}{2}-\frac12, \,\ 
z-\frac{\tau}{2}+\frac12, \,\ 
\frac{\tau}{8}\Big)
\\[2mm]
&=&
(-1)^{mp} \, 
\bigg[\sum_{\substack{k \in \zzz \\[1mm] 0 \, < \, k \, \leq \, mj}}
-
\sum_{\substack{k \in \zzz \\[1mm] mj \, < \, k \, \leq \, 0}} \bigg] 
(-1)^{(m+1)j+k} \, 
q^{(m+1)(j+\frac{1}{2(m+1)}+\frac{m(4p+1)}{4(m+1)})^2
\, - \, \frac{1}{m}(k+\frac{m(4p+1)}{4})^2} 
\nonumber
\\[-2mm]
& & \hspace{60mm}
\times \,\ 
\big[\theta_{2k,m}-\theta_{-2k,m}\big](\tau, z)
\nonumber
\\[2mm]
& & \hspace{-2mm}
+ \,\ (-1)^{mp} \, \eta(2\tau)^3 \, \bigg\{
\frac{\theta_{2p-\frac12, m+1}(\tau,z)}{\theta_{2p-\frac12, 1}(\tau,z)}
\, - \, 
\frac{\theta_{-2p+\frac12, m+1}(\tau,z)}{\theta_{-2p+\frac12, 1}(\tau,z)}
\bigg\}
\\[2mm]
& & \hspace{-2mm}
+ \,\ 
(-1)^{p(m+1)} \theta^{(\sigma(m))}_{2p-1-\frac{m}{2}, m+1}(\tau,0)
\sum_{\substack{k \, \in \zzz \\[1mm] 1 \, \leq \, k \, \leq \, pm}} 
(-1)^k  \, q^{-\frac{1}{m}(k+\frac{m}{4})^2} \, 
\big[\theta_{2k, \, m} - \theta_{-2k, \, m}\big] (\tau, z)
\end{eqnarray*}}
Multiplying \, $(-1)^{p(m+1)}$ to this equation we obtain 
the formula \eqref{n3:eqn:2022-827h1}, proving 1).

\medskip
\noindent
\underline{Proof of 2)} : \,\ Multiplying \, 
$\theta^{(\sigma(m))}_{m(2p+\frac12), m+1}(\tau,0)$ to 
${\rm \eqref{n3:eqn:2022-827j}}_{s=\frac12}$, we have 
{\allowdisplaybreaks
\begin{eqnarray*}
& & \hspace{-7mm} \hspace{-10mm}
\underbrace{\theta^{(\sigma(m))}_{m(2p+\frac12),m+1}(\tau,0)}_{
\substack{|| \\[0mm] {\displaystyle \hspace{13mm}
(-1)^{p(m+1)}
\theta^{(\sigma(m))}_{2p-\frac{m}{2}, m+1}(\tau,0)
}}} \hspace{-10.8mm} 
\Phi^{[\frac{m}{2},\frac12]}\Big(2\tau, \,\ 
z+\frac{\tau}{2}-\frac12, \,\ 
z-\frac{\tau}{2}+\frac12, \,\ 
\frac{\tau}{8}\Big)
\\[2mm]
&=&
(-1)^{mp} 
\underbrace{q^{-m(p+\frac14)^2} 
\theta^{(\sigma(m))}_{m(2p+\frac12),m+1}(\tau,0) 
\Phi^{[\frac{m}{2},\frac12]}\Big(2\tau, 
z+\frac{\tau}{2}-\frac12+2p\tau, 
z-\frac{\tau}{2}+\frac12-2p\tau, 0\Big)}_{\rm (II)} 
\\[2mm]
& & \hspace{-5mm}
+ \,\ e^{-\frac{\pi i}{2}} 
\theta^{(\sigma(m))}_{m(2p+\frac12),m+1}(\tau,0)
\sum_{\substack{k \, \in \zzz \\[1mm] 1 \, \leq \, k \, \leq \, pm}} 
(-1)^k  \, q^{-\frac{1}{m}(k-\frac12+\frac{m}{4})^2} \, 
\big[\theta_{2(k-\frac12), \, m} - \theta_{-2(k-\frac12), \, m}\big] 
(\tau, z)
\end{eqnarray*}}
Here (II) is obtained by ${\rm \eqref{n3:eqn:2022-827a}}_{p \rightarrow 2p}$, 
so we have 
{\allowdisplaybreaks
\begin{eqnarray*}
& &
(-1)^{p(m+1)}
\theta^{(\sigma(m))}_{2p-\frac{m}{2}, m+1}(\tau,0) \, 
\Phi^{[\frac{m}{2},\frac12]}\Big(2\tau, \,\ 
z+\frac{\tau}{2}-\frac12, \,\ 
z-\frac{\tau}{2}+\frac12, \,\ 
\frac{\tau}{8}\Big)
\\[2mm]
&=&
(-1)^{mp} \, \bigg[
\sum_{\substack{(j, k) \, \in \, \zzz \times \frac12 \zzz_{\rm odd} \\[1mm]
0 \, < \, k \, < \, mj}}
-
\sum_{\substack{(j, k) \, \in \, \zzz \times \frac12 \zzz_{\rm odd} \\[1mm]
mj \, < \, k \, < \, 0}}
\bigg] \, 
(-1)^{(m+1)j+k} \, q^{(m+1)(j+\frac{m(4p+1)}{4(m+1)})^2
\, - \, \frac{1}{m}(k+\frac{m(4p+1)}{4})^2} 
\nonumber
\\[-6mm]
& & \hspace{75mm}
\times \,\ \big[\theta_{2k,m}-\theta_{-2k,m}\big](\tau, z)
\nonumber
\\[3.5mm]
& & \hspace{-2mm}
- \,\ (-1)^{mp} \, i \,\ \eta(2\tau)^3 \, \bigg\{
\frac{\theta_{2p+\frac12, m+1}(\tau,z)}{\theta_{2p-\frac12, 1}(\tau,z)}
\, - \, 
\frac{\theta_{-2p-\frac12, m+1}(\tau,z)}{\theta_{-2p+\frac12, 1}(\tau,z)}
\bigg\}
\\[3mm]
& & \hspace{-2mm}
- \, i \, (-1)^{p(m+1)} \theta^{(\sigma(m))}_{2p-\frac{m}{2}, m+1}(\tau,0)
\sum_{\substack{k \in \zzz \\[1mm] 1 \, \leq \, k \, \leq \, pm}} 
(-1)^k q^{-\frac{1}{m}(k-\frac12+\frac{m}{4})^2} \, 
\big[\theta_{2k-1, \, m} - \theta_{-(2k-1), \, m}\big] (\tau, z)
\end{eqnarray*}}
Multiplying \, $(-1)^{p(m+1)}$ to this equation we obtain 
the formula \eqref{n3:eqn:2022-827h2}, proving 2).
\end{proof}

\section{The space of N=3 numerators}
\label{sec:N=3:numerator}

\subsection{The space $\Theta^{[m]}$}
\label{subsec:Theta(m)}

\medskip

We put 
$$
\begin{array}{lcl}
\ccc[[q^{\frac12}]] &:=& \ccc\text{-linear span of} \,\ 
\bigg\{q^b\sum\limits_{j=0}^{\infty}a_j(q^{\frac12})^j \,\ ; \,\ 
b \in \rrr, \,\ a_j \in \ccc\bigg\}
\\[5mm]
\ccc((q^{\frac12})) &:=& \bigg\{ \dfrac{f}{g} \,\ ; \,\ 
f, \, g \, \in \, \ccc[[q^{\frac12}]], \,\ g \ne 0 \bigg\}
\end{array}
$$
and consider the following spaces for $m \, \in \, \nnn$:
$$
\begin{array}{lcl}
\Theta^{[m]} &:=& \ccc((q^{\frac12}))\text{-linear span of} \,\ 
\big\{[\theta_{j,m}-\theta_{-j,m}](\tau,z)\big\}_{j \in \zzz}
\\[2mm]
\Theta^{[m,0]} &:=& \ccc((q^{\frac12}))\text{-linear span of} \,\ 
\big\{[\theta_{j,m}-\theta_{-j,m}](\tau,z)\big\}_{j \in \zzz_{\rm even}}
\\[2mm]
\Theta^{[m, \frac12]} &:=& \ccc((q^{\frac12}))\text{-linear span of} \,\ 
\big\{[\theta_{j,m}-\theta_{-j,m}](\tau,z)\big\}_{j \in \zzz_{\rm odd}}
\end{array}
$$

\vspace{1mm}

\begin{note} 
\label{n3:note:2022-827a}
For $m \in \nnn$ the following hold:
\begin{enumerate}
\item[{\rm 1)}] \quad $f \in \Theta^{[m]} 
\quad \Longrightarrow \quad 
\theta_{0,1} f, \,\ \theta_{1,1} f \, \in \, \Theta^{[m+1]}$
\item[{\rm 2)}] \quad $\left\{
\begin{array}{lcl}
\theta_{0,1} \cdot \Theta^{[m,0]} &\subset& \Theta^{[m+1,0]} \\[1mm]
\theta_{0,1} \cdot \Theta^{[m,\frac12]} &\subset& \Theta^{[m+1,\frac12]}
\end{array} \right. \hspace{5mm} \left\{
\begin{array}{lcl}
\theta_{1,1} \cdot \Theta^{[m,0]} &\subset& \Theta^{[m+1,\frac12]} \\[1mm]
\theta_{1,1} \cdot \Theta^{[m,\frac12]} &\subset& \Theta^{[m+1,0]}
\end{array} \right. $
\end{enumerate}
\end{note}

\begin{proof} These claims can be checked easily by calculation using 
the formula (2.2a) in \cite{W2022d}.
\end{proof}

The following Note \ref{n3:note:2022-827b} is also obtained by easy calculation.

\medskip

\begin{note} \,\ 
\label{n3:note:2022-827b}
\begin{enumerate}
\item[{\rm 1)}]
\begin{enumerate}
\item[{\rm (i)}] \quad $\theta_{0,1}(\tau, \, \tau+\frac12) \, = \, 0$
\item[{\rm (ii)}] \quad $\theta_{1,1}(\tau, \, \tau+\frac12) \, = \, 
- i \, q^{-\frac14} \, \sum\limits_{j \in \zzz} \, (-1)^jq^{j^2}$
\end{enumerate}
\item[{\rm 2)}]
\begin{enumerate}
\item[{\rm (i)}] \quad $\vartheta_{11}(\tau, \, \tau+\frac12) 
\, = \, i \, \big(\theta_{1,2}-\theta_{-1,2}\big)
\,\ \in \,\ \Theta^{[2]}$
\item[{\rm (ii)}] \quad $\vartheta_{11}(\tau, \, \tau+\frac12) 
\,\ = \,\ 
- \, 2 \, q^{-\frac38} \, \prod\limits_{n=1}^{\infty}(1-q^n)(1+q^n)^2$
\end{enumerate}
\end{enumerate}
\end{note}

\vspace{1mm}

By these Notes \ref{n3:note:2022-827a} and \ref{n3:note:2022-827b},
we can prove the following:

\begin{lemma} 
\label{n3:lemma:2022-828c}
Let $m \in \nnn_{\geq 2}$, then  
\begin{enumerate}
\item[{\rm 1)}] \quad $\Theta^{[m+1]} \,\ = \,\
\theta_{0,1}\Theta^{[m]} \, + \, \theta_{1,1}\Theta^{[m]}$
\item[{\rm 2)}]
\begin{enumerate}
\item[{\rm (i)}] \quad $\Theta^{[m+1, 0]} \,\ = \,\
\theta_{0,1}\Theta^{[m,0]} \, + \, \theta_{1,1}\Theta^{[m, \frac12]}$
\item[{\rm (ii)}] \quad $\Theta^{[m+1, \frac12]} \,\ = \,\
\theta_{0,1}\Theta^{[m, \frac12]} \, + \, \theta_{1,1}\Theta^{[m,0]}$
\end{enumerate}
\end{enumerate}
\end{lemma}

\begin{proof} 1) Consider the $\ccc((q^{\frac12}))$-linear subspace of 
$\Theta^{[m+1]}$ spanned by $\theta_{0,1}\Theta^{[m]}$ and 
$(\theta_{1,1})^{m-1}\vartheta_{11}$. We note the following:

\begin{enumerate}
\item[(a)] \,\ $
{\rm dim}_{\ccc((q^{\frac12}))}(\theta_{0,1} \cdot \Theta^{[m]})
\, = \, 
{\rm dim}_{\ccc((q^{\frac12}))}\Theta^{[m]} \, = \, m-1$ 
\quad and \quad 
${\rm dim}_{\ccc((q^{\frac12}))}\Theta^{[m+1]} \, = \, m$ 

\item[(b)] \,\ $f \, \in \, \theta_{0,1} \cdot \Theta^{[m]}
\quad \Longrightarrow \quad f(\tau, \tau+\frac12) \, = \, 0$

\item[(c)] \,\ $f \, = \, (\theta_{1,1})^{m-1}\vartheta_{11}
\quad \Longrightarrow \quad f(\tau, \tau+\frac12) \, \ne \, 0$
\end{enumerate}

\noindent
Then we have
$$
\Theta^{[m+1]} \, = \, \theta_{0,1} \cdot \Theta^{[m]}
\, \oplus \, 
\ccc((q^{\frac12})) \cdot (\theta_{1,1})^{m-1}\vartheta_{11}
\,\ \subset \,\ 
\theta_{0,1} \cdot \Theta^{[m]} + \theta_{1,1} \cdot \Theta^{[m]}
$$
proving 1). \, 2) follows immediately from 1) and 
Note \ref{n3:note:2022-827a}.
\end{proof}

\subsection{The spaces $V^{[m,0]}$ and $V^{[m,\frac12]}$}
\label{subsec:V(ms):U(ms)}

\medskip

For $m \in \nnn$ and $s \in \big\{0, \, \frac12 \big\}$, we define 
the spaces $V^{[m,s]}$ and $U^{[m,s]}$ as follows:
{\allowdisplaybreaks
\begin{eqnarray*}
& & \hspace{-7mm}
V^{[m,0]} 
\overset{\substack{def \\[1mm] }}{:=} 
\ccc((q^{\frac12}))\text{-linear span of} \,\ \bigg\{
\Phi^{[\frac{m}{2},s]}\Big(
2\tau, z+\frac{\tau}{2}-\frac12, z-\frac{\tau}{2}+\frac12, 
\frac{\tau}{8}\Big) \,\ ; \,\ 
\begin{array}{l}
s \, \in \, \zzz \\[0mm]
1 \leq s \leq \frac{m+1}{2}
\end{array}
\bigg\}
\\[3mm]
& & \hspace{-7mm}
V^{[m,\frac12]} 
\overset{\substack{def \\[1mm] }}{:=} 
\ccc((q^{\frac12}))\text{-linear span of} \,\ \bigg\{
\Phi^{[\frac{m}{2},s]}\Big(
2\tau, z+\frac{\tau}{2}-\frac12, z-\frac{\tau}{2}+\frac12, 
\frac{\tau}{8}\Big) \,\ ; \,\ 
\begin{array}{l}
s \, \in \, \frac12 \, \zzz_{\rm odd} \\[1mm]
\frac12 \leq s \leq \frac{m+1}{2}
\end{array}
\bigg\}
\\[3mm]
& & \hspace{-7mm}
U^{[m,0]} 
\overset{\substack{def \\[1mm] }}{:=} 
\ccc((q^{\frac12}))\text{-linear span of} \,\ 
\Bigg\{
\frac{\theta_{-\frac12, m+1}(\tau,z)}{\theta_{-\frac12,1}(\tau,z)}
-
\frac{\theta_{\frac12, m+1}(\tau,z)}{\theta_{\frac12,1}(\tau,z)}, 
\quad 
\underset{ \Big(\substack{\\[1mm]
k \, \in \, \nnn_{\rm even} \\[1mm]
1 \, \leq \, k \, \leq \, m-1
} \Big) }{\big[\theta_{k,m}-\theta_{-k,m}\big](\tau,z)}
\Bigg\}
\\[3mm]
& & \hspace{-7mm}
U^{[m,\frac12]} 
\overset{\substack{def \\[1mm] }}{:=} 
\ccc((q^{\frac12}))\text{-linear span of} \,\ 
\Bigg\{
\frac{\theta_{\frac12, m+1}(\tau,z)}{\theta_{-\frac12,1}(\tau,z)}
-
\frac{\theta_{-\frac12, m+1}(\tau,z)}{\theta_{\frac12,1}(\tau,z)}, 
\quad 
\underset{ \Big(\substack{\\[1mm]
k \, \in \, \nnn_{\rm odd} \\[1mm]
1 \, \leq \, k \, \leq \, m-1
} \Big) }{\big[\theta_{k,m}-\theta_{-k,m}\big](\tau,z)}
\Bigg\}
\end{eqnarray*}}
Note that 
$$ \hspace{-10mm}
\begin{array}{lcl}
{\rm dim}_{\ccc((q^{\frac12}))} \, V^{[m,0]} 
&=&
\left\{
\begin{array}{ccl}
\frac{m}{2} & & {\rm if} \,\ m \, \in \, \nnn_{\rm even} \\[2mm]
\frac{m+1}{2} & & {\rm if} \,\ m \, \in \, \nnn_{\rm odd}
\end{array} \right. 
\\[5mm]
{\rm dim}_{\ccc((q^{\frac12}))} \, V^{[m,\frac12]} 
&=& \left\{
\begin{array}{ccl}
\frac{m}{2}+1 & & {\rm if} \,\ m \, \in \, \nnn_{\rm even} \\[2mm]
\frac{m+1}{2} & & {\rm if} \,\ m \, \in \, \nnn_{\rm odd}
\end{array} \right.
\end{array}
$$

\vspace{1mm}

\begin{lemma} \quad 
\label{n3:lemma:2022-828a}
$V^{[m,s]} \,\ = \,\ U^{[m,s]}$ \quad for \quad 
$m \in \nnn$ \, and \, $s \in \{0, \, \frac12\}$.
\end{lemma}

\begin{proof} \,\ $V^{[m,s]} \, \subset \, U^{[m,s]}$ \, 
is obvious by Lemma 2.2 in \cite{W2022b} and 
Proposition \ref{n3:prop:2022-827a}.
The opposite inclusion $U^{[m,s]} \, \subset \, V^{[m,s]}$ \, is 
shown as follows.  First we note that 
$$
\Theta^{[m,0]} \,\ \subset \,\ V^{[m,0]} 
\hspace{5mm} \text{and} \hspace{5mm}
\Theta^{[m,\frac12]} \,\ \subset \,\ V^{[m,\frac12]}
$$
by Note 5.1 in \cite{W2022d}. \, Next, by Proposition \ref{n3:prop:2022-827a},
we have 
{\allowdisplaybreaks
\begin{eqnarray*}
\frac{\theta_{-\frac12, m+1}(\tau,z)}{\theta_{-\frac12,1}(\tau,z)}
- 
\frac{\theta_{\frac12, m+1}(\tau,z)}{\theta_{\frac12,1}(\tau,z)}
&\in & \ccc((q^{\frac12}))
\Phi^{[\frac{m}{2},0]}\Big(2\tau, z+\frac{\tau}{2}-\frac12, 
z-\frac{\tau}{2}+\frac12, \frac{\tau}{8}\Big) + 
\underbrace{\Theta^{[m,0]}}_{\substack{
\rotatebox{-90}{$\subset$} \\[0.5mm] {\displaystyle V^{[m,0]}
}}}
\\[-7mm]
& \subset & V^{[m.0]}
\\[2mm]
\frac{\theta_{\frac12, m+1}(\tau,z)}{\theta_{-\frac12,1}(\tau,z)}
- 
\frac{\theta_{-\frac12, m+1}(\tau,z)}{\theta_{\frac12,1}(\tau,z)}
&\in & \ccc((q^{\frac12}))
\Phi^{[\frac{m}{2},\frac12]}\Big(2\tau, z+\frac{\tau}{2}-\frac12, 
z-\frac{\tau}{2}+\frac12, \frac{\tau}{8}\Big) + 
\underbrace{\Theta^{[m,\frac12]}}_{\substack{
\rotatebox{-90}{$\subset$} \\[0.5mm] {\displaystyle V^{[m,\frac12]}
}}}
\\[-7mm]
& \subset & V^{[m.\frac12]}
\end{eqnarray*}}
These formulas give \, $U^{[m.0]} \subset V^{[m.0]}$ \, and \, 
$U^{[m.\frac12]} \subset V^{[m.\frac12]}$, proving Lemma 
\ref{n3:lemma:2022-828a}.
\end{proof}

\begin{note} 
\label{n3:note:2022-828a}
For $m \in \nnn$ and $p \in \zzz$ the following hold:
\begin{enumerate}
\item[{\rm 1)}] \,\ $
\dfrac{\theta_{2p-\frac12, m+1}(\tau,z)}{\theta_{-\frac12,1}(\tau,z)}
- 
\dfrac{\theta_{-2p+\frac12, m+1}(\tau,z)}{\theta_{\frac12,1}(\tau,z)}
\,\ \in \,\ U^{[m,0]}$
\item[{\rm 2)}] \,\ $
\dfrac{\theta_{2p+\frac12, m+1}(\tau,z)}{\theta_{-\frac12,1}(\tau,z)}
- 
\dfrac{\theta_{-2p-\frac12, m+1}(\tau,z)}{\theta_{\frac12,1}(\tau,z)}
\,\ \in \,\ U^{[m,\frac12]}$
\end{enumerate}
\end{note}

\begin{proof} These claims are clear since, in Proposition 
\ref{n3:prop:2022-827a}, 
the RHS's of the formulas \eqref{n3:eqn:2022-827h1} and 
\eqref{n3:eqn:2022-827h2} multiplied by $\ccc((q^{\frac12}))$ 
are independent of $p \in \zzz_{\geq 0}$ and then 
independent of $p \in \zzz$ due to Lemma 1.2 in \cite{W2022c}.
\end{proof}

\vspace{1mm}

\begin{lemma} \quad 
\label{n3:lemma:2022-828b}
For $m \in \nnn$ the following formulas hold:
\begin{enumerate}
\item[{\rm 1)}]
\begin{enumerate}
\item[{\rm (i)}] \quad $\theta_{0,1}(\tau,z) \, \times \, \Bigg\{
\dfrac{\theta_{-\frac12, m+1}(\tau,z)}{\theta_{-\frac12,1}(\tau,z)}
-
\dfrac{\theta_{\frac12, m+1}(\tau,z)}{\theta_{\frac12,1}(\tau,z)}\Bigg\}
\,\ \in \,\ U^{[m+1, 0]}$
\item[{\rm (ii)}] \quad $\theta_{0,1}(\tau,z) \, \times \, \Bigg\{
\dfrac{\theta_{\frac12, m+1}(\tau,z)}{\theta_{-\frac12,1}(\tau,z)}
-
\dfrac{\theta_{-\frac12, m+1}(\tau,z)}{\theta_{\frac12,1}(\tau,z)}\Bigg\}
\,\ \in \,\ U^{[m+1, \frac12]}$
\end{enumerate}
\item[{\rm 2)}]
\begin{enumerate}
\item[{\rm (i)}] \quad $\theta_{1,1}(\tau,z) \, \times \, \Bigg\{
\dfrac{\theta_{-\frac12, m+1}(\tau,z)}{\theta_{-\frac12,1}(\tau,z)}
-
\dfrac{\theta_{\frac12, m+1}(\tau,z)}{\theta_{\frac12,1}(\tau,z)}\Bigg\}
\,\ \in \,\ U^{[m+1, \frac12]}$
\item[{\rm (ii)}] \quad $\theta_{1,1}(\tau,z) \, \times \, \Bigg\{
\dfrac{\theta_{\frac12, m+1}(\tau,z)}{\theta_{-\frac12,1}(\tau,z)}
-
\dfrac{\theta_{-\frac12, m+1}(\tau,z)}{\theta_{\frac12,1}(\tau,z)}\Bigg\}
\,\ \in \,\ U^{[m+1, 0]}$
\end{enumerate}
\end{enumerate}
\end{lemma}

\begin{proof} By the formula (2.2a) in \cite{W2022d}, we have
{\allowdisplaybreaks
\begin{eqnarray*}
& &
\left\{
\begin{array}{l}
\theta_{0,1}(\tau, z) \theta_{\frac12,m+1}(\tau,z) 
= \hspace{-2mm}
\sum\limits_{\substack{\\[1mm] r \in \zzz/(m+2)\zzz}} \hspace{-5mm}
\theta_{\frac12+2r,m+2}(\tau,z) \, 
\theta_{\frac12-2(m+1)r, (m+1)(m+2)}(\tau,0)
\\[6mm]
\theta_{0,1}(\tau, z) \theta_{-\frac12,m+1}(\tau,z) 
= \hspace{-2mm}
\sum\limits_{\substack{\\[1mm] r \in \zzz/(m+2)\zzz}} \hspace{-5mm}
\theta_{-\frac12-2r,m+2}(\tau,z) \, 
\theta_{\frac12-2(m+1)r, (m+1)(m+2)}(\tau,0)
\end{array}\right.
\\[1mm]
& &
\left\{
\begin{array}{l}
\theta_{1,1}(\tau, z) \theta_{\frac12,m+1}(\tau,z) 
= \hspace{-2mm}
\sum\limits_{\substack{\\[1mm] r \in \zzz/(m+2)\zzz}} \hspace{-5mm}
\theta_{\frac12+2r+1,m+2}(\tau,z) \, 
\theta_{\frac12-(2r+1)(m+1), (m+1)(m+2)}(\tau,0)
\\[6mm]
\theta_{1,1}(\tau, z) \theta_{-\frac12,m+1}(\tau,z) 
= \hspace{-3mm}
\sum\limits_{\substack{\\[1mm] r \in \zzz/(m+2)\zzz}} \hspace{-5mm}
\theta_{-\frac12-2r-1,m+2}(\tau,z) \, 
\theta_{\frac12-(2r+1)(m+1), (m+1)(m+2)}(\tau,0)
\end{array}\right.
\end{eqnarray*}}
Using these equations and Note \ref{n3:note:2022-828a}, we have 

\medskip

\noindent
1) (i) \quad $\theta_{0,1}(\tau,z) \, \Bigg\{
\dfrac{\theta_{-\frac12, m+1}(\tau,z)}{\theta_{-\frac12,1}(\tau,z)}
-
\dfrac{\theta_{\frac12, m+1}(\tau,z)}{\theta_{\frac12,1}(\tau,z)}\Bigg\}$

$$
= \hspace{-5mm}
\sum_{r \in \zzz/(m+2)\zzz}\Bigg\{
\frac{\theta_{-\frac12-2r,m+2}(\tau,z)}{\theta_{-\frac12,1}(\tau,z)} 
-
\frac{\theta_{\frac12+2r,m+2}(\tau,z)}{\theta_{\frac12,1}(\tau,z)} 
\Bigg\} \, \theta_{\frac12-2(m+1)r, (m+1)(m+2)}(\tau,0)
\in 
U^{[m+1,0]}
$$

(ii) \quad $\theta_{0,1}(\tau,z) \, \Bigg\{
\dfrac{\theta_{\frac12, m+1}(\tau,z)}{\theta_{-\frac12,1}(\tau,z)}
-
\dfrac{\theta_{-\frac12, m+1}(\tau,z)}{\theta_{\frac12,1}(\tau,z)}\Bigg\}$

$$
= \hspace{-5mm}
\sum_{r \in \zzz/(m+2)\zzz}\Bigg\{
\frac{\theta_{\frac12+2r,m+2}(\tau,z)}{\theta_{-\frac12,1}(\tau,z)} 
-
\frac{\theta_{-\frac12-2r,m+2}(\tau,z)}{\theta_{\frac12,1}(\tau,z)} 
\Bigg\} \, \theta_{\frac12-2(m+1)r, (m+1)(m+2)}(\tau,0)
\in 
U^{[m+1,\frac12]}
$$

\noindent
2) (i) \quad $\theta_{1,1}(\tau,z) \, \Bigg\{
\dfrac{\theta_{-\frac12, m+1}(\tau,z)}{\theta_{-\frac12,1}(\tau,z)}
-
\dfrac{\theta_{\frac12, m+1}(\tau,z)}{\theta_{\frac12,1}(\tau,z)}\Bigg\}$

{\allowdisplaybreaks
\begin{eqnarray*}
&=& \hspace{-5mm}
\sum_{r \in \zzz/(m+2)\zzz}\Bigg\{
\frac{\theta_{-\frac12-2r-1,m+2}(\tau,z)}{\theta_{-\frac12,1}(\tau,z)} 
-
\frac{\theta_{\frac12+2r+1,m+2}(\tau,z)}{\theta_{\frac12,1}(\tau,z)} 
\Bigg\} \, \theta_{\frac12-(2r+1)(m+1), (m+1)(m+2)}(\tau,0)
\\[2mm]
& \hspace{-5mm}
\underset{\substack{\\[0.5mm] \uparrow \\[1mm] r \rightarrow r-1
}}{=} \hspace{-5mm} &
\sum_{r \in \zzz/(m+2)\zzz}\Bigg\{
\frac{\theta_{\frac12-2r,m+2}(\tau,z)}{\theta_{-\frac12,1}(\tau,z)} 
-
\frac{\theta_{-\frac12+2r,m+2}(\tau,z)}{\theta_{\frac12,1}(\tau,z)} 
\Bigg\} \, \theta_{\frac12-(2r-1)(m+1), (m+1)(m+2)}(\tau,0)
\\[2mm]
& \in &
U^{[m+1,\frac12]}
\end{eqnarray*}}

(ii) \quad $\theta_{1,1}(\tau,z) \, \Bigg\{
\dfrac{\theta_{\frac12, m+1}(\tau,z)}{\theta_{-\frac12,1}(\tau,z)}
-
\dfrac{\theta_{-\frac12, m+1}(\tau,z)}{\theta_{\frac12,1}(\tau,z)}\Bigg\}$

{\allowdisplaybreaks
\begin{eqnarray*}
&=& \hspace{-5mm}
\sum_{r \in \zzz/(m+2)\zzz}\Bigg\{
\frac{\theta_{\frac12+2r+1,m+2}(\tau,z)}{\theta_{-\frac12,1}(\tau,z)} 
-
\frac{\theta_{-\frac12-2r-1,m+2}(\tau,z)}{\theta_{\frac12,1}(\tau,z)} 
\Bigg\} \, \theta_{\frac12-(2r+1)(m+1), (m+1)(m+2)}(\tau,0)
\\[2mm]
& \hspace{-5mm}
\underset{\substack{\\[0.5mm] \uparrow \\[1mm] r \rightarrow r-1
}}{=} \hspace{-5mm} &
\sum_{r \in \zzz/(m+2)\zzz}\Bigg\{
\frac{\theta_{-\frac12+2r,m+2}(\tau,z)}{\theta_{-\frac12,1}(\tau,z)} 
-
\frac{\theta_{\frac12-2r,m+2}(\tau,z)}{\theta_{\frac12,1}(\tau,z)} 
\Bigg\} \, \theta_{\frac12-(2r-1)(m+1), (m+1)(m+2)}(\tau,0)
\\[2mm]
& \in &
U^{[m+1,0]}
\end{eqnarray*}}
Thus the proof of Lemma \ref{n3:lemma:2022-828b} is completed.
\end{proof}

\vspace{1mm}

\begin{prop} 
\label{n3:prop:2022-828a}
For $m \in \nnn$ the following formulas hold:
\begin{enumerate}
\item[{\rm 1)}] \quad $\theta_{0,1} \cdot V^{[m,0]} 
\, + \, \theta_{1,1} \cdot V^{[m,\frac12]} 
\,\ = \,\ V^{[m+1,0]}$
\item[{\rm 2)}] \quad $\theta_{0,1} \cdot V^{[m,\frac12]} 
\, + \, \theta_{1,1} \cdot V^{[m,0]} 
\,\ = \,\ V^{[m+1,\frac12]}$
\end{enumerate}
\end{prop}

\begin{proof} These formulas follow immediately from Lemmas 
\ref{n3:lemma:2022-828c}, \ref{n3:lemma:2022-828a} and
\ref{n3:lemma:2022-828b}. 
\end{proof}

\section{The space of N=3 characters}
\label{sec:N=3:character}

\medskip

For $m \in \nnn$, we put 
$$
\begin{array}{lcl}
\overset{N=3}{CH}{}^{[K(m), {\rm even}]} &:=& 
\ccc((q^{\frac12}))\text{-linear span of} \,\ 
\bigg\{ \overset{N=3}{\rm ch}{}^{(+)}_{H(\Lambda^{[K(m), m_2]})} \,\ ; \,\ 
\begin{array}{l}
m_2 \, \in \, \zzz_{\rm even} \\[0mm]
0 \leq m_2 \leq m
\end{array}
\bigg\}
\\[5mm]
\overset{N=3}{CH}{}^{[K(m), {\rm odd}]} &:=& 
\ccc((q^{\frac12}))\text{-linear span of} \,\ 
\bigg\{ \overset{N=3}{\rm ch}{}^{(+)}_{H(\Lambda^{[K(m), m_2]})} \,\ ; \,\ 
\begin{array}{l}
m_2 \, \in \, \zzz_{\rm odd} \\[0mm]
0 \leq m_2 \leq m
\end{array}
\bigg\}
\\[5mm]
\overset{N=3}{CH}{}^{[K(m)]} &:=& 
\overset{N=3}{CH}{}^{[K(m), {\rm even}]} \,\ \oplus \,\ 
\overset{N=3}{CH}{}^{[K(m), {\rm odd}]} 
\end{array}
$$

\vspace{1mm}

\begin{note} 
\label{n3:note:2022-828b}
For $m \in \nnn$ the following formulas hold:
\begin{enumerate}
\item[{\rm 1)}] \,\ $
\overset{N=3}{CH}{}^{[K(m), {\rm even}]} \,\ = \,\ 
\Bigg\{\dfrac{f(\tau,z)}{\overset{N=3}{R}{}^{(+)}(\tau,z)} 
\,\ ; \,\ f \in V^{[m,\frac12]}\Bigg\} 
\,\ = \,\ 
\Bigg\{\dfrac{f(\tau,z)}{\overset{N=3}{R}{}^{(+)}(\tau,z)} 
\,\ ; \,\ f \in U^{[m,\frac12]}\Bigg\} $
\item[{\rm 2)}] \,\ $
\overset{N=3}{CH}{}^{[K(m), {\rm odd}]} \,\ = \,\ 
\Bigg\{\dfrac{f(\tau,z)}{\overset{N=3}{R}{}^{(+)}(\tau,z)} 
\,\ ; \,\ f \in V^{[m,0]}\Bigg\} 
\,\ = \,\ 
\Bigg\{\dfrac{f(\tau,z)}{\overset{N=3}{R}{}^{(+)}(\tau,z)} 
\,\ ; \,\ f \in U^{[m,0]}\Bigg\} $
\end{enumerate}
\end{note}

\begin{proof} These formulas are clear from 
Proposition 4.1 in \cite{W2022a} and Lemma \ref{n3:lemma:2022-828a}.
\end{proof}

\vspace{1mm}

\begin{prop} 
\label{n3:prop:2022-828b}
For $m \in \nnn$ the following formulas hold:
\begin{enumerate}
\item[{\rm 1)}]
\begin{enumerate}
\item[{\rm (i)}] \quad $\theta_{0,1} \cdot 
\overset{N=3}{CH}{}^{[K(m), {\rm even}]} 
\, + \, \theta_{1,1} \cdot 
\overset{N=3}{CH}{}^{[K(m), {\rm odd}]}
\,\ = \,\ 
\overset{N=3}{CH}{}^{[K(m+1), {\rm even}]}$
\item[{\rm (ii)}] \quad $\theta_{0,1} \cdot 
\overset{N=3}{CH}{}^{[K(m), {\rm odd}]} 
\, + \, \theta_{1,1} \cdot 
\overset{N=3}{CH}{}^{[K(m), {\rm even}]}
\,\ = \,\ 
\overset{N=3}{CH}{}^{[K(m+1), {\rm odd}]}$
\end{enumerate}
\item[{\rm 2)}] \quad $\theta_{0,1} \cdot \overset{N=3}{CH}{}^{[K(m)]} 
\, + \, \theta_{1,1} \cdot \overset{N=3}{CH}{}^{[K(m)]}
\,\ = \,\ 
\overset{N=3}{CH}{}^{[K(m+1)]}$
\end{enumerate}
\end{prop}

\begin{proof} These formulas follow immediately from Proposition 
\ref{n3:prop:2022-828a} and Note \ref{n3:note:2022-828b}.
\end{proof}

\vspace{1mm}

Using the above Proposition \ref{n3:prop:2022-828b}, we can 
prove the following theorem:

\begin{thm} 
\label{n3:thm:2022-827a}
For $m, n \in \nnn$, the following formulas hold:
\begin{enumerate}
\item[{\rm 1)}] \quad $\overset{N=3}{CH}{}^{[K(m)]}
\cdot \overset{N=3}{CH}{}^{[K(n)]}
\,\ = \,\ 
\overset{N=3}{CH}{}^{[K(m+n)]}$
\item[{\rm 2)}] 
\begin{enumerate}
\item[{\rm (i)}] \quad $\overset{N=3}{CH}{}^{[K(m), {\rm even}]}
\cdot \overset{N=3}{CH}{}^{[K(n), {\rm even}]}
\,\ \subset \,\ 
\overset{N=3}{CH}{}^{[K(m+n), {\rm even}]}$
\item[{\rm (ii)}] \quad $\overset{N=3}{CH}{}^{[K(m), {\rm even}]}
\cdot \overset{N=3}{CH}{}^{[K(n), {\rm odd}]}
\,\ \subset \,\ 
\overset{N=3}{CH}{}^{[K(m+n), {\rm odd}]}$
\item[{\rm (iii)}] \quad $\overset{N=3}{CH}{}^{[K(m), {\rm odd}]}
\cdot \overset{N=3}{CH}{}^{[K(n), {\rm odd}]}
\,\ \subset \,\ 
\overset{N=3}{CH}{}^{[K(m+n), {\rm even}]}$
\end{enumerate}
\end{enumerate}
\noindent
where, for $\ccc((q^{\frac12}))$\text{-}{\rm spaces} $A$ and $B$, 
$A \cdot B$ denotes the linear span of \,\  
$\{ fg \, ; \, f \in A, \,\ g \in B\}$.
\end{thm}

\begin{proof} We shall prove the claim 1) by induction on $m$. 
As the 1st step, in the case $m=1$, $\overset{N=3}{CH}{}^{[K(1)]}$ is given 
by Proposition 3.1 in \cite{W2022d}: 
\begin{equation}
\overset{N=3}{CH}{}^{[K(1)]} \,\ = \,\ 
\ccc((q^{\frac12})) \, \theta_{0,1} \, + \, \ccc((q^{\frac12})) \, \theta_{1,1}
\label{n3:eqn:2022-923a}
\end{equation} 
So the claim 1) holds by Proposition \ref{n3:prop:2022-828b}.

\medskip

As the 2nd step, we shall prove that the claim 1) holds for $m+1$ 
assuming that it holds for $m$, as follows. Since 
$$
\overset{N=3}{CH}{}^{[K(m+1)]} \,\ = \,\ 
\theta_{0,1} \cdot \overset{N=3}{CH}{}^{[K(m)]} 
\, + \, \theta_{1,1} \cdot \overset{N=3}{CH}{}^{[K(m)]}
$$
by Proposition \ref{n3:prop:2022-828b}, we have
$$
\overset{N=3}{CH}{}^{[K(m+1)]}
\cdot \,\ \overset{N=3}{CH}{}^{[K(n)]}
=
\theta_{0,1} \cdot 
\underbrace{\overset{N=3}{CH}{}^{[K(m)]} \cdot \, \overset{N=3}{CH}{}^{[K(n)]}
}_{\substack{|| \\[-1mm] {\displaystyle \overset{N=3}{CH}{}^{[K(m+n)]}
}}}
\, + \, 
\theta_{1,1} \cdot 
\underbrace{\overset{N=3}{CH}{}^{[K(m)]} \cdot \, \overset{N=3}{CH}{}^{[K(n)]}
}_{\substack{|| \\[-1mm] {\displaystyle \overset{N=3}{CH}{}^{[K(m+n)]}
}}}
$$
which is equal to $\overset{N=3}{CH}{}^{[K(m+n+1)]}$
again by Proposition \ref{n3:prop:2022-828b}.
Thus we have proved 1). \\

The claim 2) follows also by induction on $m$ using Proposition 
\ref{n3:prop:2022-828b}, just in the similar way with the proof of 1).
\end{proof}

\vspace{0mm}

From this theorem, we obtain the following:

\vspace{0mm}

\begin{cor}
\label{n3:cor:2022-827a}
Let $m \in \nnn$ and $m_2 \in \zzz$ such that $0 \leq m_2 \leq m$.
Then
\begin{enumerate}
\item[{\rm 1)}] the character of the N=3 module $H(\Lambda^{[K(m), m_2]})$
can be written as a $\ccc((q^{\frac12}))$-linear combination of 
$(\theta_{1,1}(\tau,z))^j(\theta_{0,1}(\tau,z))^{m-j}$ \,\ 
$(0 \leq j \leq m)$.

\item[{\rm 2)}]
\begin{enumerate}
\item[{\rm (i)}] If $m_2 \in \zzz_{\rm even}$, \,
$\overset{N=3}{\rm ch}{}^{(+)}_{H(\Lambda^{[K(m), m_2]})}(\tau,z)$ is a 
$\ccc((q^{\frac12}))$-linear combination of \\
$(\theta_{1,1}(\tau,z))^j(\theta_{0,1}(\tau,z))^{m-j}$ \,\ 
$(j \in \zzz_{\rm even})$,
\item[{\rm (ii)}] If $m_2 \in \zzz_{\rm odd}$, \,
$\overset{N=3}{\rm ch}{}^{(+)}_{H(\Lambda^{[K(m), m_2]})}(\tau,z)$ is a 
$\ccc((q^{\frac12}))$-linear combination of \\
$(\theta_{1,1}(\tau,z))^j(\theta_{0,1}(\tau,z))^{m-j}$ \,\ 
$(j \in \zzz_{\rm odd})$.
\end{enumerate}
\end{enumerate}
\end{cor}

\begin{proof} By Theorem \ref{n3:thm:2022-827a} and induction on $m$, 
we have
$$
\overset{N=3}{CH}{}^{[K(m)]} \, = \, 
\Big(\overset{N=3}{CH}{}^{[K(1)]}\Big)^m
\, 
\underset{\substack{\\[0.5mm] \uparrow \\[1mm] \text{by \eqref{n3:eqn:2022-923a}}
}}{=} \, 
\bigoplus_{j=0}^m\ccc ((q^{\frac12})) \, (\theta_{1,1})^j(\theta_{0,1})^{m-j},
$$
proving 1). \, 2) follows from 1) and the arguments in this section.
\end{proof}

\end{document}